\documentclass[reqno,a4paper]{amsart}
\usepackage{eucal,amsfonts,amssymb,amsmath,amsthm,epsfig,mathrsfs}

\usepackage{amscd,amsxtra}
\usepackage{enumerate}
\usepackage{latexsym}

\allowdisplaybreaks

 \makeatletter \@addtoreset{equation}{section}

\makeatother \makeatletter

\title[Nonautonomous Kolmogorov equations in the whole space]
{Nonautonomous Kolmogorov equations in the whole space: a survey on
recent results}
\author{L. Lorenzi}
\address{Dipartimento di Matematica e Informatica, Universit\`a degli Studi di Parma, Parco Area delle Scienze 53/A, I-43124 Parma, Italy.}
\email{luca.lorenzi@unipr.it}

\subjclass[2000]{Primary: 47D06; Secondary: 47F05, 35B65, 37L40, 35B10, 35B40}
\keywords{nonautonomous parabolic equations, evolution operators,gradient estimates, evolution systems of measures,
evolution semigroups, invariant measures,asymptotic behaviour, compactness}

\newtheorem{theorem}{Theorem}[section]
\newtheorem{hyp}[theorem]{Hypothesis}
\newtheorem{remark}[theorem]{Remark}{\rm}
\newtheorem{proposition}[theorem]{Proposition}
\newtheorem{example}[theorem]{Example}

\newcommand\R{{\mathbb R}}

\newcommand\N{{\mathbb N}}
\renewcommand{\phi}{\varphi}

\begin{document}

\begin{abstract}
In this paper we survey some recent results concerned with nonautonomous Kolmogorov elliptic operators.
Particular attention is paid to the case of the nonautonomous Ornstein-Uhlenbeck operator
\end{abstract}

\maketitle

\section{Introduction}

The interest in elliptic operators with unbounded coefficients in
$\R^N$ and in smooth unbounded subsets has grown sensibly in the
last decades due to their applications in many branches of applied
sciences (for instance mathematical finance). Starting from the
pioneering papers by Azencott and It\^o (see \cite{azencott,ito} and
also \cite{MPW}) the study of autonomous Kolmogorov operators has
spread out and led to an almost rich literature nowadays. We refer
the reader to \cite{bertoldi-lorenzi} and its bibliography.

One of the keystone in the analysis of autonomous nondegenerate
elliptic operators is the study of the Ornstein-Uhlenbeck operator
\begin{eqnarray*}
({\mathscr A}\varphi)(x)=\sum_{i,j=1}^Nq_{ij}D_{ij}\varphi(x)+\sum_{i,j=1}^Nb_{ij}x_jD_i\varphi(x),\qquad\;\,x\in\R^N,
\end{eqnarray*}
where $Q=(q_{ij})$ and $B=(b_{ij})$ are given constant matrices, $Q$
being positive definite. Such analysis begun in the paper
\cite{daprato-lunardi} and continued in several other papers (among
them we quote
\cite{CFMP,lorenzi-dyn,lunardi-O-U,metafune-O-U,metafune-pallara-enrico,metafune-O-U-2}).
The main feature of the Ornstein-Uhlenbeck operator, which makes it
easier to be studied than more general operators with unbounded
coefficients, is an explicit representation formula for the solution
to the Cauchy problem
\begin{equation}
\left\{
\begin{array}{lll}
D_tu(t,x)=({\mathscr A}u)(t,x), &t>0, & x\in\R^N,\\[2mm]
u(0,x)=f(x), &&x\in\R^N,
\end{array}
\right.
\label{Cauchy-intro}
\end{equation}
when $f\in C_b(\R^N)$. It turns out that $u(t,x)=(T(t)f)(x)$ for any $t>0$ and any $x\in\R^N$, where
the so called Ornstein-Uhlenbeck semigroup $(T(t))$ is defined by
\begin{eqnarray*}
(T(t)f)(x):=\frac{1}{(4\pi)^{N/2}({\rm det} Q_t)^{1/2}}\int_{\R^N}e^{-\frac{1}{4}\langle Q_t^{-1}y,y\rangle}
f(y+e^{tB}x)dy,\qquad\;\,x\in\R^N,
\end{eqnarray*}
for any $f\in C_b(\R^N)$.

For more general elliptic operators ${\mathscr A}$ with unbounded
coefficients of the form
\begin{eqnarray*}
({\mathscr A}\varphi)(x)=\sum_{i,j=1}^Nq_{ij}(x)D_{ij}\varphi(x)+\sum_{i,j=1}^Nb_jD_j\varphi(x),\qquad\;\,x\in\R^N,
\end{eqnarray*}
it has been proved that, under mild assumptions on the regularity of the coefficients $q_{ij}$ and $b_j$ ($i,j=1,\ldots,N$),
the Cauchy problem \eqref{Cauchy-intro} admits {\it at least} a bounded classical solution (i.e. there exists a bounded function $u$
which belongs to $C^{1,2}((0,+\infty)\times\R^N)\cap C([0,+\infty)\times\R^N)$ and solves the Cauchy problem
\eqref{Cauchy-intro}. In this more general setting no explicit representation formula for the function $u$ is available.

As far as the nonhomogeneous Cauchy problem
\begin{eqnarray}
\left\{
\begin{array}{lll}
D_tu(t,x)=({\mathscr A}u)(t,x)+g(t,x), & t\in [0,T], & x\in\R^N,\\[1mm]
u(0,x)=f(x), && x\in \R^N,
\end{array}
\right.
\label{pb-nonhom}
\end{eqnarray}
is concerned, under suitable algebraic and growth conditions on the coefficients of the operator ${\mathscr A}$,
some Schauder type results have been proved in \cite{BL,lunardi-studia}. More specifically, in the previous papers it has been proved that
if $f\in C^{2+\theta}_b(\R^N)$, $g\in C([0,T]\times\R^N)$ and
\begin{eqnarray*}
\sup_{t\in [0,T]}\|g(t,\cdot)\|_{C^{\theta}_b(\R^N)}<+\infty,
\end{eqnarray*}
for some $\theta\in (0,1)$, then Problem \eqref{pb-nonhom} admits a unique solution $u\in C^{1,2}([0,T]\times\R^N)$ such that
\begin{eqnarray*}
\sup_{t\in [0,T]}\|u(t,\cdot)\|_{C^{2+\theta}_b(\R^N)}\le C\left (\|f\|_{C^{2+\theta}_b(\R^N)}+
\sup_{t\in [0,T]}\|g(t,\cdot)\|_{C^{\theta}_b(\R^N)}\right ),
\end{eqnarray*}
for some positive constant $C$, independent of $f$ and $g$. 

Differently from the case when the coefficients of ${\mathscr A}$
are bounded, the semigroups associated to elliptic operators with
unbounded coefficients are, in general, neither strongly continuous
in $BUC(\R^N)$, nor analytic in $C_b(\R^N)$. Moreover, the usual
$L^p$-spaces related to the Lebesgue measure are not the suitable
$L^p$-spaces where to consider Kolmogorov semigroups. A simple
one-dimensional example in \cite{PRS} shows that the operator
\begin{align*}
({\mathscr A}\varphi)(x)=\varphi''(x)-{\rm
sign}(x)|x|^{1+\varepsilon}\varphi'(x),\qquad\;\,x\in\R,
\end{align*}
does not generate a strongly continuous semigroup in $L^p(\R)$ for
whichever $\varepsilon>0$ and $p\in [1,+\infty)$.

As a matter of fact, the $L^p$-spaces which fit best the properties
of semigroups associated with elliptic operators with unbounded
coefficients are those related to a particular measure, the
so-called invariant measure of the semigroup. Such a measure, when
existing, is characterized by the following invariance property:
\begin{eqnarray*}
\int_{\R^N}T(t)f\,d\mu=\int_{\R^N}f\,d\mu,\qquad\;\,t>0,\;\,f\in C_b(\R^N).
\end{eqnarray*}
Under rather weak assumptions on the coefficients of the operator
${\mathscr A}$, if an invariant measure of $(T(t))$ exists, then it
is unique. The most famous sufficient condition ensuring the
existence of an invariant measure is the Has'minskii criterion,
which can be stated in term of a so-called Lyapunov function. More
specifically, Has'minskii criterion states that the invariant
measure exists if there exists a smooth function $\varphi$, tending to $+\infty$ as $|x|\to +\infty$,
such that ${\mathscr A}\varphi$ tends to $-\infty$ as $|x|\to +\infty$. In the
case of the Ornstein-Uhlenbeck operator, it is known that the
invariant measure exists if and only if the spectrum of the matrix
$B$ is contained in the left open halfplane $\{\lambda\in\mathbb C:
{\rm Re}\lambda<0\}$.

Whenever the invariant measure exists, the semigroup $(T(t))$ can be
extended to $L^p(\R^N,\mu)$ by a semigroup of positive contractions,
for any $p\in [1,+\infty)$, which we still denote by $(T(t))$. The
characterization of the domain of its infinitesimal generator is an
hard and challenging task, solved only in some particular situation.
This is the case, for instance, of the Ornstein-Uhlenbeck semigroup
(see \cite{lunardi-O-U,metafune-O-U-2}) where also the spectrum of
the infinitesimal generator and the sector of analyticity have been
completely characterized (see \cite{CFMP,metafune-O-U}). We also
quote the papers \cite{daprato-lunardi-1,lorenzi-lunardi} where some
more general situations are considered. In all the cases dealt with
in the previous two papers the invariant measure is explicit and
this makes the problem easier to be studied. In the general case,
the invariant measure is not explicit and only some qualitative
properties are known. It is well-known that the invariant measure is
absolutely continuous with respect to the Lebesgue measure. Under
rather weak assumptions on the smoothness of the coefficients of the
operator ${\mathscr A}$ the density of $\mu$ with respect to the
Lebesgue measure is locally H\"older continuous in $\R^N$. Global
properties of the invariant measure have been proved in
\cite{fornaro-fusco-metafune-pallara,metafune-pallara-rhandi}.

Since the characterization of the domain of the infinitesimal generator $A_p$ of the semigroup $(T(t))$ in $L^p(\R^N,\mu)$ is an
hard task in general, it turns out important to determine suitable space of smooth functions which are a core for $A_p$.
This problem has been addressed in \cite{albanese-mangino,albanese-mangino-2,albanese-mangino-lorenzi} where sufficient conditions
for $C^{\infty}_c(\R^N)$ to be a core of $A_p$ are given.

Whenever an invariant measure exists, for any $f\in L^p(\R^N,\mu)$
the function $T(t)f$ converges to the mean $\overline f$ of $f$ with
respect to $\mu$, in $L^p(\R^N,\mu)$ as $t\to +\infty$ for any $p\in
(1,+\infty)$. In particular, if the pointwise gradient estimate
\begin{eqnarray*}
|(\nabla T(t)f)(x)|^2\le Ce^{\omega t}(T(t)f^2)(x),\qquad\;\,t>1,\;\,x\in\R^N,
\end{eqnarray*}
holds true for any $f\in C_b(\R^N)$ and some constants $C>0$ and
$\omega<0$, then $T(t)f$ converges to $\overline f$ with exponential
rate.

In this paper we are going to survey the recent results in the case of nonautonomous elliptic operators with unbounded coefficients
starting from the pioneering paper \cite{daprato-lunardi-2}.

The paper is structured as follows. In Section \ref{sect-2} we
introduce the evolution operators $(P(t,s))$ associated to
nonautonomous elliptic operators
\begin{equation}
({\mathscr
A}\varphi)(s,x)=\sum_{i,j=1}^Nq_{ij}(s,x)D_{ij}\varphi(x)+\sum_{i,j=1}^Nb_j(s,x)D_j\varphi(x),\quad\;\,s\in
I,\;\,x\in\R^N, \label{op-nonauton}
\end{equation}
in $C_b(\R^N)$, where $I$ is a right halfline (possibly $I=\R$),
listing their main properties. Section \ref{sect-3} is devoted to
proving uniform estimates for the derivatives (up to the
third-order) of the function $P(t,s)f$ when $f$ belongs to spaces of
H\"older continuous functions. As a valuable consequence of such
estimates, we state an optimal regularity result in H\"older spaces
for the solution to \eqref{pb-nonhom} when ${\mathscr A}$ is a
nonautonomous operator. We then turn our attention to pointwise
gradient estimates which are extensively used in the forthcoming
sections. Section \ref{sect-4} is devoted to introducing the
nonautonomous counterpart of the concept of invariant measures: the
so called evolution systems of invariant measures, i.e., a family
$\{\mu_s: s\in\R\}$ of probability measures such that
\begin{eqnarray*}
\int_{\R^N}P(t,s)fd\mu_t=\int_{\R^N}fd\mu_s,\qquad\;\,s<t,\;\, f\in
C_b(\R^N).
\end{eqnarray*}
From Section \ref{sect-5} we confine ourselves to the case when
$I=\R$. In Section \ref{sect-5}, we introduce the evolution
semigroup $(T(t))$ associated with the evolution operator $(P(t,s))$
both in $C_b(\R^{1+N})$- and in $L^p$-spaces related to particular
Borel positive measures $\mu$ constructed starting from evolution
systems of measures. More specifically, $\mu$ is the unique Borel
measure which extends the function
\begin{eqnarray*}
(A,B)\mapsto \int_A\mu_s(B)ds,
\end{eqnarray*}
defined on Borel sets $A\subset\R$ and $B\subset\R^N$.

The semigroup $(T(t))$, defined in $C_b(\R^{1+N})$, extends to
$L^p(\R^{1+N},\mu)$ by a strongly continuous semigroup of
contractions. In the case of the nonautonomous Ornstein-Uhlenbeck
operator, and $\mu$ coming from the unique evolution
system of measures of Gaussian type, the domain of the infinitesimal
generator of the evolution semigroup $(T(t))$ is characterized in
Subsection \ref{subsect-5.1}. Section \ref{sect-6} is devoted to the
periodic case, i.e., to the case when the coefficients of the
nonautonomous operator ${\mathscr A}$ are $T$ periodic with respect
to $s$. Section \ref{sect-7} collects some results on the asymptotic
behaviour of the evolution operator $(P(t,s))$ in the $L^p$-spaces
related to evolution systems of measures. Finally, in Section
\ref{sect-8} we present some sufficient conditions, in the periodic
case, for the generator of $(T(t))$ be compactly embedded in
$L^p(\R^{1+N},\mu^{\sharp})$, where $\mu^{\sharp}$ is the
(probability) measure constructed starting from the unique $T$
periodic evolution systems of measures of $(P(t,s))$.

\subsection*{Notation}
Given an open set $\Omega\subset\R^N$ and a smooth function
$u:\Omega\to\R$,
we use the notation $D_iu$ and $D_{ij}u$ to denote the derivatives
$\frac{\partial u}{\partial x_i}$ and $\frac{\partial^2u}{\partial
x_ix_j}$, respectively.
If $u$ is a function of the variables $s$ and $x$, we denote by $D_su$ the derivative
$\frac{\partial u}{\partial s}$.

The subscript ``$b$'' means bounded. Hence, $C_b(\R^N)$ stands for
the set of all continuous functions $f:\R^N\to\R$ which are bounded.
We endow $C_b(\R^N)$ with the sup-norm. Similarly, for any $k>0$,
$C^k_b(\R^N)$ stands for the set of all functions in $C^k(\R^N)$
which are bounded and have bounded derivatives up to the $[k]$-th
order. It is endowed with the Euclidean norm
\begin{eqnarray*}
\|u\|_{C^k_b(\R^N)}=\sum_{|\alpha|\le
[k]}\|D^{\alpha}u\|_{C_b(\R^N)}+\sum_{|\alpha|=[k]} \sup_{x\neq
y}\frac{|D^{\alpha}u(x)-D^{\alpha}u(y)}{|x-y|^{k-[k]}}.
\end{eqnarray*}

The subscript ``c'' always means compactly supported. Hence, $C_c^2(\R^N)$
stands for the set of all the twice-continuously differentiable functions with compact
support in $\R^N$.

By $B_R$ we denote the open ball in ${\mathbb R}^N$ with centre at
the origin and radius $R$, and by $\overline B_R$ its closure. If
$A$ is a measurable set in $\R^N$, we denote by $\chi_A$ the
characteristic function of the set $A$. Finally, by $\langle
x,y\rangle$ we denote the Euclidean inner product of the vectors
$x,y\in\R^N$.

\section{The evolution operator in $C_b(\R^N)$}
\label{sect-2}
\setcounter{equation}{0}

In this section we assume the following assumptions on the
coefficients of the operator ${\mathscr A}$ in \eqref{op-nonauton}.
\begin{hyp}\label{hyp1}
~\par
\begin{enumerate}[\rm (i)]
\item
The coefficients $q_{ij}$ and $b_i$  belong to  $C^{\alpha/2,\alpha}_{\rm loc}(I\times\R^N)$ for any
$i,j=1,\ldots,N$ and some $\alpha\in (0,1)$;
\item
$Q$ is uniformly elliptic, i.e., for every $(s,x)\in I \times \R^N$, the matrix $Q(s,x)$ is symmetric and there
exists a function $\eta:I\times\R^N\to \R$ such that $0<\eta_0:=\inf_{I\times\R^N}\eta$ and
\[
\langle Q(s,x)\xi , \xi \rangle \geq \eta(s,x) |\xi |^2, \qquad\;\,
\xi \in \R^N,\;\,(s,x) \in I \times \R^N;
\]
\item
for every bounded interval $J \subset I$ there
exist a function $\varphi = \varphi_J \in C^2(\R^N)$ and a positive
number $\lambda = \lambda_J$ such that
\[
\lim_{|x| \to +\infty}\varphi (x) = +\infty \quad \mbox{and}\quad
({\mathscr A}\varphi) (s,x) - \lambda \varphi (x) \leq 0,\qquad\;\,(s,x) \in
J\times \R^N.
\]
\end{enumerate}
\end{hyp}

Under the previous set of assumptions one can prove the following
result.

\begin{theorem}[{\cite[Theorem 2.2]{kunze-lorenzi-lunardi}}]
\label{thm:2.2}
For any $s \in I$ and any $f \in C_b(\R^N)$, there
exists a unique solution $u$ of the Cauchy problem
\begin{equation}
\left\{
\begin{array}{lll}
D_tu(t,x)=({\mathscr A}u)(t,x), &t>s, & x\in\R^N,\\[2mm]
u(s,x)=f(x), &&x\in\R^N.
\end{array}
\right. \label{Cauchy-sect-2}
\end{equation}
Furthermore,
\begin{equation}
\|u(t,\cdot)\|_{\infty} \leq
\|f\|_{\infty},\qquad\;\, t \geq s.
\label{contractive}
\end{equation}
\end{theorem}

\begin{proof}
Uniqueness and Estimate \eqref{contractive} follow from a
generalized maximum principle. Hypothesis \ref{hyp1}(iii) allows to
prove that, if $u\in C_b([a,b]\times\R^N)\cap
C^{1,2}((a,b]\times\R^N)$ satisfies the differential inequality
$D_tu-{\mathscr A}u\le 0$ in $(a,b]\times\R^N$ and $u(a,\cdot)\le
0$, then $u(t,x)\le 0$ for any $(t,x)\in [a,b]\times\R^N$. It
suffices to observe that $u$ is the pointwise limit of the sequence
of functions $v_n=u-n^{-1}\varphi$ which have a global maximum in
$[a,b]\times\R^N$, which should be non positive since $v_n$ is
non positive at $t=s$.

The existence part is obtained looking at the solution $u$ to
\eqref{Cauchy-sect-2} as the limit (in a suitable sense) of the
solutions to Cauchy-Dirichlet problems in balls.

First one considers the case when $f$ is positive and belongs to $C^{2+\alpha}_c(\R^N)$.
For any $n\in\N$, let $u_n$ be the classical solution to the Cauchy-Dirichlet problem
\begin{equation}
\left\{
\begin{array}{lll}
u_t(t,x)=({\mathscr A}u)(t,x), & t\in (s, +\infty ), &x\in B_n,\\[1mm]
u(t,x)=0, & t\in (s,+\infty), &x \in
\partial B_n,\\[1mm]
u(s,x)=f(x), && x \in B_n.
\end{array}\right.
\label{approx-dirichlet}
\end{equation}
If $n_0$ is such that ${\rm supp}(f)\subset B(n_0)$, then, for any
$n\ge n_0$, the unique classical solution to Problem
\eqref{approx-dirichlet} belongs to $C^{1+\alpha/2,2+\alpha}_{\rm
loc}([s,+\infty)\times \overline{B_n})$. Moreover, for any $m>n_0$,
there exists a constant $C=C(m)$ independent of $n$, such that
\begin{eqnarray*}
\|u_n\|_{C^{1+\alpha/2,2+\alpha}((s,m)\times B_m)} \leq C\|f\|_{C^{2+\alpha}_b(\R^N)},
\end{eqnarray*}
for any $n>m$. The sequence $(u_n(x))$ is increasing for any
$x\in\R^N$, by the classical maximum principle. Hence, the previous
estimate and a diagonal argument imply that $u_n$ converges in
$C^{1,2}((s,m)\times B(m))$, for any $m\in\N$, to some function $u\in
C^{1+\alpha/2,2+\alpha}_{\rm loc}([s,+\infty)\times\R^N)$. Clearly,
$u$ satisfies the differential equation in \eqref{Cauchy-sect-2}
since any function $u_n$ does in $(s,+\infty)\times B_n$. Moreover,
$u(s,\cdot)=f$ since $u_n(s,\cdot)=f$ for any $n\in\N$ and $u_n$
converges to $u$ locally uniformly in $[s,+\infty)\times\R^N$.

In the case when $f\in C_0(\R^N)$, one fixes a sequence
$(f_n)\subset C^{2+\alpha}_c(\R^N)$ converging to $f$ uniformly in
$\R^N$ as $n$ tends to $+\infty$. Denote by $u_{f_n}$ the solution
to \eqref{Cauchy-sect-2} with $f$ being replaced by $f_n$. Estimate \eqref{contractive} yields
\begin{eqnarray*}
\|u_{f_n}-u_{f_m}\|_{C_b([s,+\infty)\times\R^N)}\le \|f_n-f_m\|_{C_b(\R^N)},\qquad\;\,m,n\in\N.
\end{eqnarray*}
Therefore, $u_{f_n}$ converges to some function $u\in
C_b([s,+\infty)\times\R^N)$, uniformly in $[s,+\infty)\times\R^N$.
In particular, $u(s,\cdot)=f$. The classical interior Schauder
estimates applied to the sequence $(u_{f_n})$ show that $u_{f_n}$
actually converges in $C^{1,2}_{\rm loc}((s,+\infty)\times\R^N)$ to
$u$. Hence, $u$ is the bounded classical solution of Problem
\eqref{Cauchy-sect-2}.

The general case when $f\in C_b(\R^N)$ is a bit trickier to be
handled with. Let $(f_n)\in C^{2+\alpha}_c(\R^N)$ converge to $f$
locally uniformly in $\R^N$ as $n$ tends to $+\infty$. Again the
interior Schauder estimates show that, up to a subsequence,
$u_{f_n}$ converges in $C^{1,2}_{\rm loc}((s,+\infty)$
$\times\R^N)$ to some function $u\in C^{1+\alpha/2,2+\alpha}_{\rm
loc}((s,+\infty)\times\R^N)$, as $n$ tends to $+\infty$. In
particular, $u$ solves the differential equation in
\eqref{Cauchy-sect-2}.

To prove that $u$ is continuous up to $t=s$ and $u(s,\cdot)=f$, we employ a localization argument. We fix a compact set
$K\subset\R^N$ and a smooth and compactly supported function $\varphi$
such that $0\le\varphi\le 1$ and $\varphi\equiv 1$ in $K$. Further, we split $u_{f_n}=u_{\varphi f_n}+
u_{(1-\varphi)f_n}$, for any $n\in \R$.
Since the function $\varphi f$ is compactly supported in $\R^N$,
$u_{\varphi f_n}$ converges to $u_{\varphi f}$ uniformly in $[s,+\infty)\times\R^N$.

Let us now consider the sequence $(u_{(1-\varphi)f_n})$. Fix $m \in\N$. A comparison argument shows that
\begin{eqnarray*}
|(u_{(1-\varphi ) f_m})(t,x)| \leq (1 - u_{\varphi}(t,x))M,\qquad\;\,(t,x)\in (s,+\infty)\times\R^N,
\end{eqnarray*}
where $M=\sup_{n\in\N}\|f_n\|_{\infty}$. Since $u_{f_n}$ converges pointwise to $u$, for any $(t,x)\in (s,+\infty)\times\R^N$ we have
\[
|u(t,x)-f(x)| =\lim _{n\to +\infty} |u_{ f_n}(t,x)-f( x)|,\qquad\;\,(t,x)\in (s,+\infty)\times\R^N,
\]
and, for each $n\in \N$, we have
\begin{align*}
    |u_{ f_n}(t,x)-f(x)| &\le |u_{ \varphi f_n}(t,x) -f(x)| + |u_{(1-\varphi)f_n}(t,x)|\nonumber\\
&\le |u_{\varphi f_{n}}(t,x)-f(x)|+
(1 - u_{\varphi}(t,x))M.
\end{align*}
Letting $n\to +\infty$ gives
\begin{align*}
    |u(t,x)-f(x)|\le |u_{\varphi f}(t,x)-f(x)|+ (1 -
    u_{\varphi}(t,x))M,
\end{align*}
Hence, $u$ can be continuously extended  up to $t=s$ setting
$u(s,\cdot)=f$.

The general case when $f$ is not everywhere nonnegative then follows splitting $f=f^+-f^-$ where
$f^+=\max\{f,0\}$ and $f^{-}=\max\{-f,0\}$, and applying the above results to $f^+$ and $f^-$.
This completes the proof.
\end{proof}

The previous theorem allows to associate an evolution operator
$(P(t,s))$ with the operator ${\mathscr A}$. For any $f\in
C_b(\R^N)$, $P(t,s)f$ is the value at $t$ of the unique bounded
classical solution to Problem \eqref{Cauchy-sect-2}.

\begin{remark}
{\rm In the case of the nonautonomous Ornstein-Uhlenbeck operator
\begin{eqnarray*}
({\mathscr
A}_O\varphi)(s,x)=\sum_{i,j=1}^Nq_{ij}(s)D_{ij}\varphi(x)+
\sum_{i,j=1}^Nb_{ij}(s)x_jD_i\varphi(x),\qquad\;\,s\in\R,\;\,x\in\R^N,
\end{eqnarray*}
an explicit representation formula for the associated evolution
operator $(P_O(t,s))$ is known. More precisely, for any $f\in
C_b(\R^N)$ one has
\begin{equation}
(P_O(t,s)f)(x)=\frac{1}{(4\pi)^{N/2}({\rm
det}Q_{t,s})^{1/2}}\int_{\R^N}e^{-\frac{1}{4}\langle
Q_{t,s}^{-1}y,y\rangle}f(y+U(s,t)x)dy, \label{nonaut-OU}
\end{equation}
where $U(\cdot,s)$ is the solution of the problem
\begin{align}
\left\{
\begin{array}{ll}
D_tU(t,s)=-B(t)U(t,s),\;\, & t\in\R, \\[2mm]
U(s,s)={\rm Id},
\end{array}
\right.
\label{pb-per-U}
\end{align}
and $Q_{t,s}$ is the positive definite matrix defined by
\begin{eqnarray*}
Q_{t,s}=\int_s^tU(s,\xi)Q(\xi)U(s,\xi)^*d\xi,\qquad\;\,s,t\in\R,\;\,s<t.
\end{eqnarray*}

Note that $P_O(\cdot,s)f$ is the unique bounded classical solution to
Problem \eqref{Cauchy-sect-2}, just assuming that $q_{ij}$ and
$b_{ij}$ are in $C_b(\R)$ for any $i,j=1,\ldots,N$, namely, no
local H\"older regularity is required.}
\end{remark}

The evolution operator $(P(t,s))$ enjoys the following properties.

\begin{proposition}
[{\cite[Propositions 2.4 \& 3.1]{kunze-lorenzi-lunardi}}]
 \label{prop-luca} The following properties hold true.
\begin{enumerate}[\rm (i)]
\item
For any $(t,s)\in I\times I$ such that $t>s$ and any $x\in\R^N$, there exists a unique probability measure $p_{t,s}(x,dy)$ such that
\begin{equation}
(P(t,s)f)(x)=\int_{\R^N}f(y)p_{t,s}(x,dy).
\label{repres}
\end{equation}
\item
Each operator $P(t,s)$ can be extended to the set of all bounded
Borel functions through Formula \eqref{repres}. In particular, for
any Borel set $A$ with positive Lebesgue measure,
$(P(t,s)\chi_A)(x)>0$ for any $x\in\R^N$ and any $t>s$.
\item
Let $(f_n) \subset C_b(\R^N)$ be a bounded sequence and $f\in
C_b(\R^N)$. Then:
\begin{enumerate}[\rm (a)]
\item
if $f_n$ converges pointwise to $f$, then $P(\cdot , s )f_n$
converges to $P(\cdot , s )f$ locally uniformly in
$(s,+\infty)\times \R^N$;
\item
 if $f_n$ converges locally uniformly in $\R^N$ to $f$, then $P(\cdot,s)f_n$
converges to $P(\cdot , s )f$ locally uniformly in
$[s,+\infty)\times\R^N$.
\end{enumerate}
\end{enumerate}
\end{proposition}

\begin{remark}
{\rm In general, $P(t,s)$ does not transform the local uniform
convergence of $f_n$ to $f$ in uniform convergence of $P(t,s)f_n$ to
$P(t,s)f$ as $n\to +\infty$. Consider for instance the
one dimensional Ornstein-Uhlenbeck operator
\begin{eqnarray*}
({\mathscr
A}_O\varphi)(x)=\varphi''(x)+bx\varphi'(x),\qquad\;\,x\in\R.
\end{eqnarray*}
In this case $P_O(t,s)f$ is given by \eqref{nonaut-OU} with
$U(t,s)=e^{-(t-s)b}$ and
\begin{eqnarray*}
Q_{t,s}=\frac{e^{2b(t-s)}-1}{2b},\qquad\;\,t,s\in\R.
\end{eqnarray*}
Let $f_n(x)=e^{i n^{-1}x}$ for any $x\in\R$ and any $n\in\N$. $f_n$
converges to $1$ locally uniformly in $\R^N$ as $n\to +\infty$. A
straightforward computation shows that
\begin{eqnarray*}
(P_O(t,s)f_n)(x)=\exp\left (-Q_{t,s}n^{-2}\right
)e^{(i/n)e^{(t-s)b}x},\qquad\;\,x\in\R,
\end{eqnarray*}
which clearly does not converge uniformly in $\R^N$ to $P_O(t,s)1=1$
as $n\to +\infty$.}
\end{remark}

One important issue is the continuity of $P(t,s)$ with respect to
the variable $s$. Under Hypothesis \ref{hyp1} the function $P(t,s)f$
turns out to be continuously differentiable with respect to $s$ in
$I\cap (-\infty,t]$ for any $f\in C^2(\R^N)$ constant outside a
compact set, and
\begin{eqnarray*}
\frac{d}{ds}P(t,s)f=-P(t,s){\mathscr A}f,\qquad\;\,t>s.
\end{eqnarray*}

In the case when $f$ is just bounded and continuous, the continuity
of the function $s\mapsto P(t,s)f$ can be proved if we replace
Hypothesis \ref{hyp1}(iii) with the following stronger condition.
\begin{hyp}\label{hyp5}
For every bounded interval $J \subset I$ there
exist a function $\varphi = \varphi_J \in C^2(\R^N)$
diverging to $+\infty$ as $|x|$ tends to $+\infty$, and a positive constant $M_J$
such that
\[
({\mathscr A}\varphi)(s,x)\le M_J,\qquad s\in J,\;\,x\in\R^N.
\]
\end{hyp}

Under this additional assumption, one can prove the following
result, which improves Property (iii) of Proposition
\ref{prop-luca}.

\begin{proposition}[{\cite[Proposition 3.6]{kunze-lorenzi-lunardi}}]
\label{prop-2.5} Let $(f_n)$ be a bounded sequence in $C_b(\R^N)$,
such that $\|f_n\|_{\infty} \leq M$ for each $n\in \N$ and $f_n$
converges to $f\in C_b(\R^N)$ locally uniformly in $\R^N$. Then, the
function $P(\cdot,\cdot)f_n$ converges to $P(\cdot,\cdot)f$ locally
uniformly in $\Lambda\times\R^N$, where $\Lambda=\{(t,s)\in I\times
I: s\le t\}$.
\end{proposition}

Clearly, the previous proposition implies the continuity of the
function $(t,s,x)\mapsto (P(t,s)f)(x)$ in $\Lambda\times\R^N$, since
this function is continuous when $f\in C^2_c(\R^N)$, and any $f\in
C_b(\R^N)$ is the local uniform limit of a sequence of functions in
$C^2_c(\R^N)$, which is bounded in the sup-norm.

\section{Uniform and gradient estimates and optimal Schauder estimates}
\label{sect-3} \setcounter{equation}{0} Theorem \ref{thm:2.2} shows
that the function $P(t,s)f$ is twice continuously differentiable
with respect to the spatial variables in $(s,+\infty)\times\R^N$ but
provides us with no information about the boundedness of such
derivatives. For the analysis of the long time behaviour of the
function $P(t,s)f$ and of the nonhomomogeneous Cauchy problem
associated with the operator ${\mathscr A}$, uniform and pointwise
estimates for the spatial derivatives of the function $P(t,s)f$,
when $f\in C_b(\R^N)$, are of particular interest. As in the
autonomous case, they can be proved under stronger assumptions on
the coefficients of the operator ${\mathscr A}$ than Hypothesis
\ref{hyp1}.

\subsection{Uniform estimates}

Uniform gradient estimates can be proved under some algebraic
conditions on the drift coefficients $b_j$ and some growth
conditions on the diffusion coefficients $q_{ij}$
($i,j=1,\ldots,N$). More precisely, assume the following additional
condition of the coefficients of the operator ${\mathscr A}$.
\begin{hyp}\label{hyp2}
~
\par
\begin{enumerate}[\rm (i)]
\item
The coefficients $q_{ij}$ and $b_i$ ($i,j=1,\ldots,N$) and their first-order
spatial derivatives belong to $C^{\alpha/2,\alpha}_{\rm
loc}(I\times\R^N)$;
\item
there exists a locally upperly bounded function $r: I\times\R^N
\to\R$ such that
\[
\langle \nabla_x b(s,x) \xi, \xi \rangle \leq r(s,x) |\xi
|^2,\qquad\;\,\xi\in\R^N,\;\,(s,x)\in I\times\R^N;
\]
\item
there exists a locally bounded function  $\zeta : I \to [0,
+\infty)$ such that, for every $i,j,k \in \{1, \ldots , N\}$, we
have
\begin{align*}
|D_kq_{ij}(s,x)| \leq\zeta (s) \eta (s,x),\qquad\;\,(s,x)\in
I\times\R^N.
\end{align*}
\end{enumerate}
\end{hyp}

Under Hypotheses \ref{hyp1} and \ref{hyp2} the following result holds true.

\begin{theorem}[{\cite[Theorem 4.1]{kunze-lorenzi-lunardi}}]
Let $s\in I$ and
$T > s$. Then, there exist positive constants $C_1,C_2$, depending on
$s$ and $T$, such that:
\begin{enumerate}[\rm (i)]
\item
for every $f \in C^1_b(\R^N)$ we have
\begin{equation}
\|\nabla P(t,s)f\|_{\infty} \leq C_1\|f\|_{C^1_b(\R^N)}, \qquad\;\,
s<t\leq T; \label{gradientC1-C1}
\end{equation}
for every $f \in C_b(\R^N)$ we have
\begin{equation}
\|\nabla P(t,s)f\|_{\infty} \leq
\frac{C_2}{\sqrt{t-s}}\|f\|_{\infty}, \qquad\;\, s< t \leq
T.
\label{gradientC0-C1}
\end{equation}
\end{enumerate}
\end{theorem}

The previous estimates show that the spatial gradient of $P(t,s)f$
satisfies estimates similar to those holding in the case when
$P(t,s)$ is associated with an elliptic operator with smooth and
bounded coefficients.

\begin{remark}
{\rm In the case when the functions $r$ and $\xi$ in Hypothesis
\ref{hyp2} are upperly bounded in $I\times\R^N$, the constants $C_1$
and $C_2$ are independent of $s\in I$. Therefore, Estimates \eqref{gradientC1-C1}
and \eqref{gradientC0-C1} can be extended to any $t>s$.
Indeed, if $t-s>T$, we split
$P(t,s)f=P(t,t-T)P(t-T,s)f$ and estimate
\begin{eqnarray*}
\|\nabla P(t,s)f\|_{\infty}\le \frac{C_2}{\sqrt{T}}\|P(t-T,s)f\|_{\infty}\le \frac{C_2}{\sqrt{T}}\|f\|_{\infty},
\end{eqnarray*}
since $P(t-T,s)$ is a contraction.
Hence,
\begin{eqnarray*}
\|\nabla P(t,s)f\|_{\infty}\le C_3\max\{1,(t-s)^{-\frac{1}{2}}\}\|f\|_{\infty},\qquad\;\,t>s\in I,\;\,f\in C_b(\R^N),
\end{eqnarray*}
and
\begin{eqnarray*}
\|\nabla P(t,s)f\|_{\infty}\le C_4\|f\|_{C^1_b(\R^N)},\qquad\;\,t>s\in I,\;\,f\in C_b^1(\R^N),
\end{eqnarray*}
for some positive constants $C_3$ and $C_4$, independent of $s$ and $t$.}
\end{remark}

Uniform estimates for second- and third-order spatial derivatives of
the function $P(t,s)f$ can be proved under stronger assumptions.
More precisely, assume that
\begin{hyp}
\label{ipos-1} ~
\par
\noindent
\begin{enumerate}[\rm (i)]
\item
the coefficients $q_{ij},b_j$ $(i,j=1,\ldots,N$) are thrice
continuously differentiable with respect to the spatial variables in
$I\times\R^N$ and they belong to
$C^{\delta/2,\delta}(J\times B_R)$ for some $\delta\in (0,1)$, any $J\subset I$
and any $R>0$, together with their first-, second- and third-order
spatial derivatives;
\item
there exist locally bounded positive functions $C_1,C_2:I\to\R$ such
that
\begin{align*}
&|Q(s,x)x|+{\rm Tr}(Q(s,x))\le C_1(s)(1+|x|^2)\eta(s,x),\\[1mm]
&\langle b(s,x),x\rangle\le C_1(1+|x|^2)\eta(s,x),
\end{align*}
for any $s\in I$ and any $x\in\R^N$;
\item
there exist three locally bounded functions $K_1,K_2,K_3:I\to\R_+$ such
that
\begin{align*}
&|D^{\beta}q_{ij}(s,x)|\le K_{|\beta|}(s)\eta(s,x),\\
&\sum_{h,k,l,m=1}^ND_{lm}q_{hk}(s,x)\xi_{hk}\xi_{lm}\le
K_2(s)\eta(s,x)\sum_{h,k=1}^N \xi_{hk}^2,
\end{align*}
for any $i,j=1,\ldots,N$, any $|\beta|=1,3$, any $N\times N$
symmetric matrix $\Xi=(\xi_{hk})$ and any $(s,x)\in I\times\R^N$;
\item
there exist two functions $d,r:I\times\R^N\to\R$ and locally
bounded functions $L_1,L_2:I\to\R$ such that
\begin{align*}
&\langle \nabla_xb(s,x)\xi,\xi\rangle\le r(s,x)|\xi|^2,\\
&|D^\beta b_j (s,x)|\le d(s,x),\\
&r(s,x)+L_1(s)d(s,x)\le L_2(s)\eta(s,x),
\end{align*}
for any $s\in I$, any $|\beta|=2,3$, any $j=1,\ldots,N$ and any
$x,\xi\in\R^N$.
\end{enumerate}
\end{hyp}

Under this set of assumptions, in \cite[Theorem 2.4]{lorenzi-DCDS}
it has been proved that for any $h,k=0,1,2,3$, with $h\le k$ and any
$T>0$ there exists a positive constant $C=C(s,h,k,T)$ such that
\begin{equation}
\|P(t,s)f\|_{C^k_b(\R^N)}\le C(t-s)^{-\frac{k-h}{2}}\|f\|_{C^h_b(\R^N)},
\qquad\;\,f\in C_b^h(\R^N),
\label{stimasem}
\end{equation}
for any $t\in (s,s+T]$.

The proof follows the same lines as in the autonomous case and is based on the Bernstein method (see \cite{B0}).
We sketch the main ideas in the case when $h=0$ and $k=3$.
For any $n\in\N$, let $\vartheta_n:\R^N\to\R$ be the radial function defined by
$\vartheta_n(x)=\psi(|x|/n)$ for any $x\in\R^N$, where $\psi$ is a smooth
nonincreasing function such that $\chi_{[0,1/2]}\le
\psi\le\chi_{[0,1]}$. We fix $s\in I$ and define the
function
\begin{align*}
v_n(t,x)=&|u_n(t,x)|^2+a(t-s)\vartheta_n^2(x)|\nabla_xu_n(t,x)|^2+a^2(t-s)^2\vartheta_n^4(x)|D^2_xu_n(t,x)|^2\\
&+a^3(t-s)^3\vartheta_n^6(x)|D^3_xu_n(t,x)|^2,
\end{align*}
for any $t\in (s,T]$ and any $x\in B_n$, where $u_n$ is the
(unique) classical solution of the Dirichlet Cauchy problem
\eqref{approx-dirichlet} with $f$ being replaced by
$\vartheta_nf$. The positive parameter $a$ will be fixed later on.

Function $v_n$ converges pointwisely as $n\to +\infty$ to the function $v$ defined by
\begin{align*}
v(t,x)=&|(P(t,s)f)(x)|^2+a(t-s)|(\nabla_xP(t,s)f)(x)|^2+a^2(t-s)^2|(D^2_xP(t,s)f)(x)|^2\\
&+a^3(t-s)^3|(D^3_xP(t,s)f)(x)|^2,
\end{align*}
for any $(t,x)\in (s,+\infty)\times\R^N$, as $n\to +\infty$.

To prove Estimates \eqref{stimasem} it suffices to show that the
constant $a$ can be fixed, independently of $n$ such that $v_n\le
\|f\|_{\infty}$ in $[s,s+T]\times B_n$ for any $n\in\N$. This
property is obtained employing the classical maximum principle. The
function $v_n$ is smooth in $(s,+\infty)\times B_n$ and it vanishes
on $(s,+\infty)\times\partial B_n$ since $u_n$ and $\vartheta_n$ do.
(This is the reason why the function $\vartheta_n$ is introduced in the
definition of $v_n$.) Moreover, $v_n$ can be extended by continuity
up to $t=s$ setting $v_n(s,\cdot)=|\vartheta_n f|^2$. Using Hypothesis
\ref{ipos-1} one can show that the constant $a$ can be fixed
(independently of $n$) such that $D_tv_n-{\mathscr A}v_n\le 0$ in
$(s,s+T]\times B_n$. The classical maximum principle then yields
$v_n\le \|\vartheta_nf\|_{\infty}\le \|f\|_{\infty}$ as desired.

To prove \eqref{stimasem} with $h=1,2$ and $k=3$, it suffices to
apply the above argument to the functions
\begin{align*}
v_n(t,x)=&|u_n(t,x)|^2+a\vartheta_n^2(x)|\nabla_xu_n(t,x)|^2+a^2(t-s)\vartheta_n^4(x)|D^2_xu_n(t,x)|^2\\
&+a^3(t-s)^2\vartheta_n^6(x)|D^3_xu_n(t,x)|^2
\end{align*}
and
\begin{align*}
v_n(t,x)=&|u_n(t,x)|^2+a\vartheta_n^2(x)|\nabla_xu_n(t,x)|^2+a^2\vartheta_n^4(x)|D^2_xu_n(t,x)|^2\\
&+a^3(t-s)\vartheta_n^6(x)|D^3_xu_n(t,x)|^2,
\end{align*}
respectively.

\begin{remark}
{\rm As for the gradient estimates,
if the functions $C_i$, $L_i$ ($i=1,2$), $K_j$ ($j=1,2,3$), $d$  and $r$ are globally upperly bounded in $I$ and
$I\times\R^N$, respectively, then
Estimates \eqref{stimasem} can be extended to any $t>s$, up to replacing
$C(t-s)^{-\frac{k-h}{2}}$ with $\tilde C\max\{(t-s)^{-\frac{k-h}{2}},1\}$, for some constant $\tilde C$, independent of $s$ and $t$.}
\end{remark}

\subsection{Optimal Schauder estimates}

Estimates \eqref{stimasem} are the keystone to prove optimal
regularity results for the nonhomogeneous Cauchy problem associated
with the operator ${\mathscr A}$. The following result holds true.

\begin{theorem}
\label{thm-schauder} Fix $[a,b]\subset I$, $\theta\in (0,1)$, $g\in
C^{0,\theta}([a,b]\times\R^N)$ and $f\in C^{2+\theta}_b(\R^N)$.
Then, the Cauchy problem
\begin{equation}
\left\{
\begin{array}{lll}
D_tu(t,x)=(\mathscr{A}u)(t,x)+g(t,x),\quad &t\in [a,b], &x\in\R^N,\\[1.5mm]
u(a,x)=f(x),&&x\in\R^N,
\end{array}
\right.
\label{non-aut-optimal}
\end{equation}
admits a unique bounded classical solution. Moreover, $u(t,\cdot)\in
C^{2+\theta}_b(\R^N)$ for any $t\in [a,b]$ and there exists a
positive constant $C$ such that
\begin{eqnarray*}
\sup_{t\in [a,b]}\|u(t,\cdot)\|_{C^{2+\theta}(\R^N)}\le C\left (
\|f\|_{C^{2+\theta}_b(\R^N)}+\sup_{t\in
[a,b]}\|g(t,\cdot)\|_{C^{\theta}_b(\R^N)}\right ).
\end{eqnarray*}
\end{theorem}

\begin{remark}
{\rm The Cauchy problem \eqref{non-aut-optimal} has been considered in
\cite{lorenzi-lumer} also in some situation where the coefficients
of the operator ${\mathscr A}$ are not smooth. More specifically, in
\cite{lorenzi-lumer} the case when the operator ${\mathscr A}$ is
given by
\begin{eqnarray*}
({\mathscr A}\varphi)(s,x)
=\sum_{i,j=1}^Nq_{ij}(s,x)D_{ij}\varphi(x)+\sum_{i,j=1}^Nb_{ij}(s)x_jD_i\varphi(x)+\sum_{j=1}^Nc_j(s,x)D_j\varphi(x),
\end{eqnarray*}
for any $s\in [0,T]$ and any $x\in\R^N$, has been considered under
the following set of assumptions.

\begin{hyp}
\par
\noindent
\begin{enumerate}[\rm (i)]
\item
The coefficients $c_j$ and $q_{ij}=q_{ji}:[0,T]\times\R^N\to\R$
($i,j=1,\ldots,N$) are measurable. Moreover, for any $s\in [0,T]$
the functions $c_j(s,\cdot)$ and $q_{ij}(s,\cdot)$ belong to
$C^{\theta}_b(\R^N)$ for some $\theta\in (0,1)$ and
\begin{eqnarray*}
\;\;\;\;\;\;\;\;\;\;\sup_{s\in [0,T]}\|c_j(s,\cdot)\|_{C^{\theta}_b(\R^N)}+\sup_{s\in
[0,T]}\|q_{ij}(s,\cdot)\|_{C^{\theta}_b(\R^N)}<+\infty,\qquad\;\,i,j=1,\ldots,N,
\end{eqnarray*}
\item
there exists $\eta_0>0$ such that
$\sum_{i,j=1}^Nq_{ij}(s,x)\xi_i\xi_j\ge\eta_0 |\xi|^2$, for
any $s\in {\mathscr D}$ and any $x,\xi\in\R^N$, where ${\mathscr D}$
is a measurable set, whose complement is negligible in $[0,T]$;
\item
the coefficients $b_{ij}$ are bounded and measurable in $[0,T]$ for
any $i,j=1,\ldots,N$.
\end{enumerate}
\end{hyp}

Assume that $f\in C^{2+\theta}_b(\R^N)$ and $g$ is a bounded and
measurable function, everywhere defined in $[0,T]\times\R^N$, such
that $g(t,\cdot)\in C^{\theta}_b(\R^N)$ for any $t\in [0,T]$ and
\begin{eqnarray*}
\sup_{t\in [0,T]}\|g(t,\cdot)\|_{C^{\theta}_b(\R^N)}<+\infty.
\end{eqnarray*}
Then, in \cite[Theorem 1.2]{lorenzi-lumer} it has been proved that
there exists a unique function $u$ such that
\begin{enumerate}[\rm (i)]
\item
$u$ is Lipschitz continuous in $[0,T]\times B_R$ for any $R>0$, its
first- and second-order spatial derivatives are bounded and continuous
functions in $[0,T]\times\R^N$;
\item
$u(0,x)=f(x)$ for any $x\in\R^N$;
\item
there exists a set ${\mathscr F}\subset [0,T]\times\R^N$, with negligible
complement, such that $D_tu(t,x)=({\mathscr A}u)(t,x)+g(t,x)$ for any
$(t,x)\in {\mathscr F}$. Moreover, for any $x\in\R^N$, the set
${\mathscr F}(x)=\{t\in [0,T]: (t,x)\in {\mathscr F}\}$ is
measurable with measure $T$.
\end{enumerate}
}
\end{remark}

\subsection{Pointwise gradient estimates}
Pointwise gradient estimates play a particular role in the study of
the properties of the evolution operator $P(t,s)$. By pointwise
gradient estimates we mean any estimate of the type
\begin{equation}
|(\nabla_x P(t,s)\varphi)(x)|^p \leq e^{p\ell_p(t-s)}(P(t,s)|\nabla
\varphi |^p)(x), \qquad\;\, t>s,\;\,x\in\R^N; \label{grad-punt}
\end{equation}
if $\varphi\in C^1_b(\R^N)$ and
\begin{equation}
|(\nabla_x P(t,s)\varphi)(x)|^p \leq C_p \max\{(t-s)^{-p/2}, \,1\}
e^{p\ell_p(t-s)}(P(t,s)|\varphi |^p)(x), \label{grad-punt-1}
\end{equation}
if $\varphi\in C_b(\R^N)$, for any $s,t\in I$, with $s<t$, any $p>1$, and some
constants $C_p>0$ and $\ell_p\in\R$.

Such estimates have been proved in \cite[Theorem 4.5]{kunze-lorenzi-lunardi} and \cite[Theorem 2.6]{lorenzi-lunardi-zamboni}
under Hypothesis \ref{hyp1} and

\begin{hyp}\label{hyp10}
\par
\noindent
\begin{enumerate}[\rm (i)]
\item
The first-order spatial derivatives of the coefficients $q_{ij}$ and $b_i$ $(i,j=1,\ldots,N)$ exist and belong to $C^{\alpha /2, \alpha}_{\rm loc}(I\times \R^N)$;
\item
Hypotheses \ref{hyp2}(ii)-(iii) are satisfied for some upperly
bounded functions $r:I\times\R^N\to\R$ and $\zeta:I\to\R_+$.
\item
the function
\begin{align*}
(s,x)\mapsto r(s,x) +\frac{N^3(\zeta(s))^2\eta(s,x)}{4\min\{p -1,1\}}
\end{align*}
\end{enumerate}
is upperly bounded in $I\times\R^N$.
\end{hyp}

\noindent The constant $\ell_p$ in \eqref{grad-punt} and
\eqref{grad-punt-1} is
\begin{equation}
\ell_p=\sup_{(s,x)\in I\times\R^N}\left (r(s,x)
+\frac{N^3(\zeta(s))^2\eta(s,x)}{4\min\{p -1,1\}}\right ).
\label{ellp}
\end{equation}

\section{Evolution systems of invariant measures}
\setcounter{equation}{0}
\label{sect-4}

Evolution systems of invariant measures (also called \emph{entrance
laws at $-\infty$} in \cite{dynkin}) are the nonautonomous
counterpart of invariant measure. By definition an evolution system
of invariant measures is a one parameter family of probability
measures $\{\mu_s: s\in I\}$ such that
\begin{equation}
\int_{\R^N}P(t,s)fd\mu_t=\int_{\R^N}fd\mu_s,
\label{esim}
\end{equation}
for any $s,t\in I$, with $s<t$ and any $f\in C_b(\R^N)$.

A sufficient condition ensuring the existence of an evolution system
of invariant systems is a variant of the Has'minskii criterion of
the autonomous case. More precisely,

\begin{theorem}[{\cite[Theorem 5.4]{kunze-lorenzi-lunardi}} see also {\cite[Theorem 3.1]{DPR1}}]
Under Hypotheses $\ref{hyp1}(i)$-$(ii)$, suppose that there exist a
positive function $\varphi\in C^2(\R^N)$ blowing up as $|x|\to
+\infty$, positive constants $a$ and $c$, and $s_0\in I$ such that
\begin{align}
({\mathscr A}\varphi)(s,x)\le a-c\varphi(x),\qquad\;\,(s,x)\in
(s_0,+\infty)\times\R^N.
\label{cond-mis-inv}
\end{align}
Then, there exists an evolution system of invariant measure of
$(P(t,s))$.
\end{theorem}

\begin{example}
{\rm Condition \eqref{cond-mis-inv} is satisfied, for instance, in the case when
the operator ${\mathscr A}$ is defined on smooth functions $\varphi$ by
\begin{eqnarray*}
({\mathscr A}\varphi)(s,x) = \Delta \varphi(x) + \sum_{j=1}^Nb_j(s,x)D_j\varphi(x),
\end{eqnarray*}
under the following assumptions on $b=(b_1,\ldots,b_N)$.
\begin{hyp}
~\par
\begin{enumerate}[\rm (i)]
\item
The functions $b_j$ $(j=1,\ldots,N)$ and their first-order spatial derivatives belong to $C^{\alpha/2,\alpha}_{\rm loc}(I\times\mathbb R^N)$ for some $\alpha\in (0,1)$;
\item
the function $b(\cdot,0)$ is bounded in $I$;
\item
there exists a continuous function $C:I\to\R$ such that
\begin{enumerate}[\rm (a)]
\item
$C$ is bounded from above in $I$;
\item
$\limsup_{t\to +\infty}C(t)<0$;
\item
$\langle \nabla_xb(t,x)\xi,\xi\rangle\le C(t)|\xi|^2$ for any $t\in I$, and $x,\xi\in\R^N$.
\end{enumerate}
\end{enumerate}
\end{hyp}

A straightforward computation reveals that, for any $m\in\N$, the function $\varphi:\R^N\to\R$, defined by $\varphi(x)=1+|x|^{2m}$ for any $x\in\R^N$, satisfies Condition \eqref{cond-mis-inv} for some $s_0\in I$.}
\end{example}

The main difference with the classical Has'minskii criterion is that
this latter just requires that the function ${\mathscr A}\varphi$
tends to $-\infty$ as $|x|\to +\infty$ without any condition on the
way it diverges to $-\infty$.

It is worth noting that Condition \eqref{cond-mis-inv} is assumed
only in a neighborhood of $+\infty$ and not in the whole of $I$.
Indeed, if the family $\{\mu_s: s\in I\}$ satisfies \eqref{esim},
then $\mu_s=P(t,s)^*\mu_t$ for any $s<t$, $s\in I$, where $P(t,s)^*$
denotes the adjoint to the operator $P(t,s)$. Hence, the measures
$\mu_s$ are uniquely determined by $\mu_t$ through the evolution
operator. The main issue is, thus, the proof of the existence of
$\mu_t$ for $t$ large.

We mention that the existence of an evolution system of invariant
measures has been proved also in \cite{BDPR08}, under different
assumptions on the coefficients of the operator ${\mathscr A}$,  and
in \cite{DPR}, for a class of nonautonomous elliptic operators,
obtained by perturbing the drift coefficients of an autonomous
Ornstein-Uhlenbeck operator by a function $F:\R^{1+N}\to\R$, which
is, roughly speaking, Lipschitz continuous in $x$ uniformly with
respect to $s$ and of dissipative type.

As a matter of fact, the evolution systems of invariant measures are
infinitely many in general, this being in contrast to the autonomous
case where the invariant measure is unique whenever the semigroup
$(T(t))$ associated with the autonomous operator ${\mathscr A}$ is
strong Feller and irreducible (properties that $(T(t))$ fulfills
under very weak assumptions on the coefficients of the operator
${\mathscr A}$).

In the case when ${\mathscr A}_O$ is the nonautonomous Ornstein-Uhlenbeck operator
\begin{align*}
({\mathscr A}_O\varphi)(s,x) =  \sum_{i,j=1}^N q_{ij}(s)D_{ij}\varphi
(x)+ \sum_{i,j=1}^N b_{ij}(s)x_jD_i\varphi(x), \qquad
(s,x)\in\R^{1+N},
\end{align*}
where $Q=(q_{ij})$ is uniformly positive definite,
Geissert and Lunardi in \cite[Proposition 2.2]{geissert-lunardi}
have proved the existence of an evolution system of invariant measures
in the case when
there exist positive constants $C_0$ and $\omega$ such that
\begin{equation}
\|U(t,s)\|_{L(\R^N)}\le C_0e^{-\omega(s-t)},\qquad\;\,s,t\in\R,\;\,s\ge t,
\label{cond-U}
\end{equation}
where $U(\cdot,s)$ solves the Cauchy problem \eqref{pb-per-U}.
Actually, Geissert and Lunardi define the nonautonomous
Ornstein-Uhlenbeck operator as the operator $(G_O(t,s))$ naturally
associated with the Cauchy problem
\begin{equation}
\left\{
\begin{array}{lll}
D_tu(t,x)+({\mathscr A}_Ou)(t,x)=0, &t<s, &x\in\R^N,\\[1mm]
u(s,x)=f, &&x\in\R^N,
\end{array}
\right.
\label{pb-indietro}
\end{equation}
i.e., $G_O(t,s)f$ is the value at $t$ of the unique solution to
\eqref{pb-indietro}. But a straightforward change of variables
allows to transform Problem \eqref{pb-indietro} into an initial
value problem of the form \eqref{Cauchy-intro}. If we denote by
$(P_O(t,s))$ the evolution operator associated with Problem
\eqref{Cauchy-intro}, all the results in \cite{daprato-lunardi-2}
can be rephrased for the operator $P_O(t,s)$ just observing that
\begin{eqnarray*}
P_O(t,s)f=G_O(-t,-s)f,\qquad\;\,t>s,\;\,f\in C_b(\R^N).
\end{eqnarray*}
where $G_O(t,s)$ is the evolution operator solving the Cauchy problem
\eqref{pb-indietro}, the operator ${\mathscr A}$ being defined by
\begin{eqnarray*}
({\mathscr A}_O\varphi)(s,x)=\sum_{i,j=1}^Nq_{ij}(-s)D_{ij}\varphi(x)+\sum_{i,j=1}^Nb_{ij}(-s)x_jD_i\varphi(x),\qquad\;\,s\in\R,\;\,x\in\R^N.
\end{eqnarray*}
on smooth functions $\varphi$.

Condition \eqref{cond-U} is essentially optimal since in the
autonomous case, $U(t,s)=e^{-(t-s)B}$ and \eqref{cond-U} is
equivalent to saying that the spectrum of $B$ lies in the left-hand
plane, which is the necessary and sufficient condition for the
Ornstein-Uhlenbeck semigroup have an invariant measure.

Under Condition \eqref{cond-U} Geissert and Lunardi characterized
all the evolution systems of invariant measures. To state more
precisely their result, we recall that for any probability measure
$\mu$, its Fourier transform $\hat\mu$ is defined as follows:
\begin{eqnarray*}
\hat\mu(h)=\int_{\R^N}e^{i\langle
x,h\rangle}\mu(dx),\qquad\;\,h\in\R^N.
\end{eqnarray*}
Moreover, we set
\begin{align*}
Q_s=\int_s^{+\infty}U(s,\xi)Q(\xi)U(s,\xi)^*d\xi,\qquad s\in\R.
\end{align*}
Then,

\begin{theorem}[{\cite[Proposition 2.2 and Lemma 2.3]{geissert-lunardi}}]
\label{thm-charact}
Fix $t_0\in\R$ and let $\mu$ be a probability
measure in $\R^N$. Further, let $\{\mu_t: t\in\R\}$ be the family of
probability measures defined through its Fourier transform, by
\begin{eqnarray*}
\hat\mu_t(h)=\hat\mu(U^*(t,t_0)h),\qquad\;\,t\in\R,\;\,h\in\R^N.
\end{eqnarray*}
Let $\{\nu_t: t\in\R\}$  be the family of measures defined, through
its Fourier transform, by
\begin{equation}
\hat\nu_t(h)=\exp\left (-\frac{1}{2}\langle Q_th,h\rangle\right
)\hat\mu_t(h),\qquad\;\,t>0,\;\,h\in\R^N. \label{express-mut}
\end{equation}
If $\{\nu_t: t\in\R\}$ is an evolution system of invariant
measure of $(P(t,s))$, then it has the form \eqref{express-mut}.

Finally, there exists a unique evolution system of invariant
measures with finite moments of some/any order, i.e. there exists a
unique family $\{\mu_s: s\in I\}$ of invariant measure such that
\begin{eqnarray*}
\sup_{s\in\R}\int_{\R^N}|x|^p\mu_s(dx)<+\infty,
\end{eqnarray*}
for some/any $p>0$. For any $s\in\R$, it holds that
\begin{align}
\mu_s(dx)= (4\pi)^{-\frac{N}{2}}(\det Q_s)^{- \frac{1}{2}}
 \,e^{-\frac{1}{4}\langle Q_s^{-1}x,x\rangle},\qquad s\in\R,\
 x\in\R^N.
 \label{gaussian-family}
 \end{align}
\end{theorem}

For more general nonautonomous Kolmogorov operators, in
\cite[Theorem 5.6]{kunze-lorenzi-lunardi} we have proved the
counterpart of the last statement of Theorem \ref{thm-charact}.

\begin{theorem}\label{asymptotics}
Assume that there exists $\omega < 0$ such that
\begin{eqnarray*}
\|\nabla P(t,s)f\|_{\infty} \leq Ce^{\omega (t-s)}\|f\|_{\infty},
\end{eqnarray*}
for all $t \geq s+1$, all $f \in C_b(\R^N)$ and some positive constant $C$.
Then, there exists at most one evolution system of invariant measure $\{\mu_t: t\in\R\}$
such that $\lim_{t\to +\infty} \mu_t (p)  e^{ \omega pt} =0$ for some $p>0$.
\end{theorem}

It is worth noticing that the previous theorem is in complete
agreement with the case of the Ornstein-Uhlenbeck operator. Indeed,
Condition \eqref{cond-U} implies that the Ornstein-Uhlenbeck
evolution operator $P_O(t,s)$ satisfies the pointwise gradient
Estimates \eqref{grad-punt} and \eqref{grad-punt-1} for any $p>1$
with $\ell_p=\omega$.

Let's go back to the fundamental Formula \eqref{esim}. Using Jensen
inequality and \eqref{repres} one can show that $|P(t,s)f|^p\le
P(t,s)|f|^p$ for any $s<t$ and any $f\in C_b(\R^N)$. Hence, using
\eqref{esim} one gets
\begin{equation}
\int_{\R^N}|P(t,s)f|^pd\mu_t\le \int_{\R^N}P(t,s)|f|^pd\mu_t=\int_{\R^N}|f|^pd\mu_s.
\label{star}
\end{equation}
Since $C_b(\R^N)$ is dense in $L^p(\R^N,\mu_s)$, the above formula
shows that each operator $P(t,s)$ can be extended to a contraction
from $L^p(\R^N,\mu_s)$ to $L^p(\R^N,\mu_t)$.

Note that, even if for different values of $t$ and $s$ the measures
$\mu_t$ and $\mu_s$ are equivalent (since they both are equivalent
to the Lebesgue measure), the spaces $L^p(\R^N,\mu_s)$ and
$L^p(\R^N,\mu_t)$ are different, in general. This makes the study of
the evolution operator $(P(t,s))$ in these $L^p$-spaces much more
difficult than in the autonomous case where $\mu_s\equiv\mu$ for any
$s\in\R$ and the semigroup $(T(t))$ maps $L^p(\R^N,\mu)$ into
itself. We go back to this point in Section \ref{sect-7}.

Since $\mu_t$ is a probability measure, $L^p(\R^N,\mu_t)$ contains all the bounded measurable functions.
A complete characterization of $L^p(\R^N,\mu_t)$ is out of scope since the measure $\mu_t$ is, in general, not explicit.
It is thus very important to determine suitable (unbounded) functions which belong to $L^p(\R^N,\mu_t)$.
As a matter of fact, if $\{\mu_t: t\in I\}$ is the evolution system of measures constructed in \cite[Theorem 5.4]{kunze-lorenzi-lunardi},
then the function $\varphi$ in \eqref{cond-mis-inv} is $L^1(\R^N,\mu_t)$ for any $t\ge s_0$. Moreover,
\begin{eqnarray*}
\sup_{t\ge s_0}\int_{\R^N}\varphi\, d\mu_t<+\infty.
\end{eqnarray*}
Hence, any function $f$ whose modulus can be controlled from above by  $C\varphi^{1/p}$, for a suitable positive constant
$C$, is in $L^p(\R^N,\mu_t)$ for any $t\ge s_0$.

\section{The evolution operator and the evolution semigroup in suitable $L^p$-spaces}
\setcounter{equation}{0}
\label{sect-5}

From now on, we assume that $I=\R$.
Moreover, we assume that Hypothesis \ref{hyp1} and Condition \eqref{cond-mis-inv} are satisfied.

As in the classical case (see e.g., \cite{chicone-latushkin}), it is
natural to introduce a semigroup of linear operators associated with
the operator $P(t,s)$. It is defined by
\begin{equation}
(T(t)f)(s,x)=(P(s,s-t)f(s-t,\cdot))(x),\qquad\;\,t>0,\;\,(s,x)\in\R^{1+N},
\label{sem-evol}
\end{equation}
for any $f\in C_b(\R^{1+N})$. Clearly, each operator $T(t)$ is a
contraction in $C_b(\R^{1+N})$. Note that $(T(t))$ agrees with the
semigroup of the translations when restricted to functions which are
independent of $x$. It follows that $(T(t))$ always fails to be
strongly continuous in $C_b(\R^{1+N})$. Moreover, it is neither
strong Feller nor irreducible. This means that $T(t)$ does not
improve the regularity of the datum $f$. More precisely, it does not
improve the regularity with respect to $s$ and it does not
transform nonnegative functions in strictly positive functions.
(Note that since $P(t,s)f\in C^2(\R^N)$ for any $f\in
C_b(\R^N)$ and any $t>s$, the function $T(t)f$ is twice continuously
differentiable in $\R^{1+N}$ with respect to the spatial variables,
for any $f\in C_b(\R^{1+N})$.)

Even if $(T(t))$ is not strongly continuous, one can associate an
infinitesimal generator (the so-called weak generator) $G_{\infty}$
to it, as in the case of semigroups associated with autonomous
elliptic operator. There are two equivalent ways to define the weak generator.
The first way, the more abstract one, consists in observing that the
family of bounded operators $\{R(\lambda): \lambda>0\}$, defined by
\begin{eqnarray*}
(R(\lambda)f)(s,x)=\int_0^{+\infty}e^{-\lambda t}(T(t)f)(s,x)dt,\qquad\;\,(s,x)\in\R^{1+N},
\end{eqnarray*}
for any $f\in C_b(\R^{1+N})$, satisfies the resolvent identity and
each operator of the family is injective. Hence, $\{R(\lambda):
\lambda>0\}$ is the resolvent family associated with some closed
operator, which we call the weak generator of $(T(t))$. A more
``concrete'' way to introduce $G_{\infty}$ (which is closer to the
definition of the infinitesimal generator of a strongly continuous
semigroup) is to define it as follows: $f\in D(G_{\infty})$ if and
only if
\begin{eqnarray*}
\sup_{t\in (0,1]}\left\|\frac{T(t)f-f}{t}\right\|_{\infty}<+\infty,
\end{eqnarray*}
and there exists $g\in C_b(\R^{1+N})$ such that
$\frac{T(t)f-f}{t}$ converges to $g$ as $t\to 0^+$ pointwise in $\R^{1+N}$. In this case $G_{\infty}f=g$.

$D(G_{\infty})$ turns out to be the maximal domain of the realization of the operator ${\mathscr G}:={\mathscr A}-D_s$
in $C_b(\R^{1+N})$. More precisely,

\begin{theorem}[{\cite[Theorem 2.8]{lorenzi-zamboni}}]
Under Hypothesis $\ref{hyp1}$
\begin{equation}
D(G_{\infty})= \bigg\{ \psi\in
\bigcap_{p<+\infty}W^{1,2}_p((-R,R)\times B_R) ~{\rm for~any}~R>0:
\psi,~{\mathscr G}\psi\in C_b(\R^{1+N})\bigg\}. \label{DGinfty}
\end{equation}
\end{theorem}

Starting from an evolution system $\{\mu_s: s\in\R\}$ of invariant
measures of $(P(t,s))$, one can define a positive measure $\mu$ on
the $\sigma$-algebra of the Borel sets of $\R^{1+N}$ by extending
the map
\begin{equation}
\mu (A\times B) := \int_A \mu_s(B)ds,
\label{constr-meas-mu}
\end{equation}
defined on Borel sets $A\subset
\R$ and  $B\subset\R^N$.

Note that the function $s\mapsto\mu_s(B)$ is measurable. Indeed, the
remark after Proposition \ref{prop-2.5} shows that the function
$s\mapsto (P(t,s)f)(x)$ is bounded and continuous in $(-\infty,t)$,
for any $x\in\R^N$ and any $f\in C_b(\R^N)$, and Condition
\eqref{cond-mis-inv} is stronger than Hypothesis \ref{hyp5}. Hence,
the function
\begin{eqnarray*}
s\mapsto\int_{\R^N}(P(t,s)f)(x)\,\mu_t(dx),
\end{eqnarray*}
is continuous as well in $(-\infty,t)$.
Since
\begin{eqnarray*}
\mu_s(B)=\int_{\R^N}(P(t,s)\chi_B)(x)\,\mu_t(dx),
\end{eqnarray*}
and $\chi_B$ is the pointwise limit of a bounded sequence  $(f_n)\subset C_b(\R^N)$,
the measurability of the function $s\mapsto\mu_s(B)$ follows.

$\mu$ is not a probability measure since $\mu(\R^{1+N})=+\infty$. Anyway, to some extent we still can
call it an invariant measure.
Indeed,
\begin{equation}
\int_{\R^{1+N}}T(t)fd\mu=\int_{\R^{1+N}}fd\mu,\qquad\;\,t>0,
\label{invariant}
\end{equation}
for any $f\in C_c(\R;C_b(\R^N))$. Moreover,
\begin{equation}
\int_{\R^{1+N}}{\mathscr G}\varphi d\mu=0,\qquad\;\,\varphi\in
C^{1,2}_c(\R^{1+N}), \label{subinvariant}
\end{equation}
see \cite[Lemma 6.3]{kunze-lorenzi-lunardi}.

Whenever existing a solution to \eqref{subinvariant} is locally H\"older continuous.
More precisely,
\begin{theorem}[{\cite[Theorem 3.8]{BKR01}}]
Let Hypothesis $\ref{hyp1}$ be satisfied. Suppose that $\mu$ is a positive measure satisfying \eqref{subinvariant}.
Then, $\mu$ is absolutely continuous with respect to the Lebesgue measure and its density
$\varrho$ satisfies the following properties:
\begin{enumerate}[\rm (i)]
\item
$\varrho$ is locally $\gamma$-H\"older continuous in $\R^{1+N}$ for any $\gamma \in
(0,1)$ and it is everywhere positive in $\R^{1+N}$ $($the positivity of the density
follows from the Harnack inequality in
\cite[Theorem 3]{AS}$)$;
\item
the function $\varrho$ belongs to $W^{0,1}_p((-T,T)\times B_R)$ for any $1 \leq p < +\infty$ and any
$R,T>0$.
\end{enumerate}
\end{theorem}

We stress that the previous theorem has been proved by Bogachev, Krylov and R\"ockner under weaker assumptions than those in
Hypothesis \ref{hyp1}.

Assuming much more regularity on the coefficients of the operator ${\mathscr A}$ we can improve the regularity of the function
$\varrho$. More precisely,

\begin{theorem}[{\cite[Theorem 4.2]{lorenzi-zamboni}}]
Besides Hypotheses $\ref{hyp1}$ assume that
$q_{ij}\in C^{\alpha/2,2+\alpha}_{\rm loc}(\R^{1+N})$ and
$b_j\in C^{\alpha/2,1+\alpha}_{\rm loc}(\R^{1+N})$ for any $i,j=1,\ldots,N$.
Then, the function $\varrho$ belongs to $C^{1+\alpha/2, 2+\alpha}_{\rm loc}(\R^{1+N})$.
\end{theorem}

Using \eqref{star}, \eqref{invariant} and the density of
$C^{\infty}_c(\R^{1+N})$ into $L^p(\R^{1+N},\mu)$, it can be easily
checked that the semigroup $(T(t))$ can be extended to
$L^p(\R^{1+N},\mu)$ by a strongly continuous semigroup of
contractions, for any $p\in [1,+\infty)$, which we still denote by
$(T(t))$. Its infinitesimal generator $G_p$ turns out to extend the
operator ${\mathscr G}$ defined on $C^{\infty}_c(\R^{1+N})$.

Formula \eqref{subinvariant} shows that $\mu$ is a solution to the
equation ${\mathscr G}^*\mu=0$ in the sense of distributions, where
${\mathscr G}^*$ is the adjoint to the operator ${\mathscr G}$. We
mention that such an equation has been extensively studied in these
last years by several authors (see e.g.
\cite{BDPR08,BDPRS,BRS06,BRS08}). In all these papers the authors
are concerned with the case when the whole space $\R^{1+N}$ is
replaced by $(0,1)\times\R^N$ or, more generally, by
$(a,b)\times\R^N$ for some $a,b\in\R$ such that $a<b$ (but some of
the results in the above papers apply also to the case of the whole
of $\R^{1+N}$). They look for families of probability measures $\{\mu_s: s\in (a,b)\}$
such that the measure $\mu$ defined according to
\eqref{constr-meas-mu} satisfies the equation
${\mathscr G}^*\mu=0$ and the initial condition
\begin{eqnarray*}
\lim_{t\to a}\int_{\R^N}\zeta\, d\mu_t=\int_{\R^N}\zeta\,
d\overline\mu,
\end{eqnarray*}
holds true for any $\zeta\in C^{\infty}_c(\R^N)$ and some probability measure
$\overline\mu$.

\subsection{Characterization of the domain of the generator of the $(T_O(t))$ in $L^p(\R^{1+N},\mu)$ and an optimal regularity result}
\label{subsect-5.1}
As in the autonomous case the characterization
of the domain of $G_p$ is an hard task and, at the best of our
knowledge, this problem has been solved only in the case of the
nonautonomous Ornstein-Uhlenbeck operator, first, in
\cite{geissert-lunardi} for $p=2$ and, then, in
\cite{geissert-lorenzi-schnaubelt} in the general case. In the
previous papers the measure $\mu$ is defined through formula
\eqref{constr-meas-mu} where the family $\{\mu_s : s\in\R\}$ is
defined by \eqref{gaussian-family}.

\begin{theorem}
Let Condition \eqref{cond-U} be satisfied. Then, for any $p\in (1,+\infty)$, the operator $G_p$ has domain
\begin{align*}
D(G_p)&=\{u\in L^p(\R^{1+N},\mu): D_su,D_i u, D_{ij}u\in L^p(\R^{1+N},\mu),
\ \forall\; i,j=1,\ldots,N\}\notag\\
&=:W^{1,2}_p(\R^{1+N},\mu).
\end{align*}
Moreover, $G_pu={\mathscr G}u$ for any $u\in D(G_p)$.
\end{theorem}

The characterization of the domain of $G_p$ can be rephrased into an optimal regularity result for the equation
\begin{align}\label{ou-line}
D_s u(s,\cdot)= ({\mathscr A}_O-\lambda)u(s,\cdot) + f(s,\cdot), \qquad
s\in\R,\;\,\lambda>0,
\end{align}
i.e., if $f\in L^p(\R^{1+N},\mu)$, Equation \eqref{ou-line} admits a
unique solution $u$, which belongs to $W^{1,2}_p(\R^{1+N},\mu)$.

In the case $p=2$, the characterization of $D(G_2)$ is the keystone
to prove the following optimal regularity result for the Cauchy
problem
\begin{equation}
\left\{
\begin{array}{lll}
D_su(s,x)=({\mathscr A}_Ou)(s,x)+g(s,x), &s\in (T_1,T_2), &
x\in\R^N,\\[1mm]
u(T_1,x)=f(x),
\end{array}
\right.
\label{nonaut-Lp}
\end{equation}
in $L^p$-spaces.
More precisely,

\begin{theorem}[{\cite[Theorem 1.3]{geissert-lunardi}}]
\label{thm-optimal-L2}
Fix $T_1,T_2\in\R$ such that $T_1<T_2$, $f\in W^{1,2}(\R^N,\mu_{T_1})$
and $g\in L^2((T_1,T_2)\times\R^N,\mu)$. Then, the Cauchy problem \eqref{nonaut-Lp} admits a
unique solution $u\in W^{1,2}_2((T_1,T_2)\times\R^N,\mu)$. Moreover, there exists a positive constant $C$, independent of
$f$ and $g$, such that
\begin{eqnarray*}
\|u\|_{W^{1,2}_2((T_1,T_2)\times\R^N,\mu)}\le C\left (\|f\|_{W^{1,2}(\R^N,\mu_{T_1})}+\|g\|_{L^2((T_1,T_2)\times\R^N,\mu)}\right ).
\end{eqnarray*}
\end{theorem}

The argument in the proof of the previous theorem cannot be straightforwardly extended to the case $p\neq 2$. Hence,
extending Theorem \ref{thm-optimal-L2} to the general case $p\neq 2$ is still an open problem.

\subsection{Cores of $G_p$}
For more general operators only some partial characterization of
$D(G_p)$ is known. In the case when the pointwise gradient Estimates
\eqref{grad-punt} are satisfied the following result holds true.

\begin{theorem}[{\cite[Theorem 3.4]{lorenzi-zamboni}}]
\label{prop-part-car} Suppose that $\ell_p$ is finite $($see
\eqref{ellp}$)$. Then, $D(G_p)$ is continuously embedded into
$W^{0,1}_p(\R^{1+N},\mu)=\{u\in L^p(\R^{1+N},\mu): \nabla_xu\in
(L^p(\R^{1+N},\mu))^N\}$  and there exist two positive constants
$C=C(p)$ and $\lambda_0=\lambda_0(p)$ such that
\begin{equation}
\|\,|\nabla_xu|\,\|_{L^p(\R^{1+N},\mu)}\le
C\|u\|_{L^p(\R^{1+N},\mu)}^{\frac{1}{2}} \|\lambda_0
u-G_pu\|_{L^p(\R^N,\mu)}^{\frac{1}{2}},
\label{stima-J12}
\end{equation}
for any $u\in D(G_p)$. If $\ell_p<0$, then Estimate
\eqref{stima-J12} holds true with $\lambda_0=0$.
\end{theorem}

Theorem \ref{prop-part-car} can be rephrased saying that
$W^{0,1}_p(\R^{1+N},\mu)$ belongs to the class $J_{1/2}$ between
$L^p(\R^{1+N},\mu))$ and $D(G_p)$.

Due to the difficulty in characterizing the domain of $G_p$, it
turns out to be extremely important to determine suitable cores for
the operator $G_p$, in order to deal with such an operator. Some
positive answers to this problem have been given in
\cite{lorenzi-zamboni}.  More precisely,

\begin{theorem}[{\cite[Theorem 2.1]{lorenzi-zamboni}}]
\label{prop-core} Let Hypotheses $\ref{hyp1}$ and Condition
$\ref{cond-mis-inv}$ be satisfied. Then, the set
\begin{align*}
D_{\rm comp}({\mathscr G})=&\bigg\{
\psi\in C_b(\R^{1+N})\cap W^{1,2}_p((-R,R)\times B_R)~{\rm for~any}~R>0,\,p<+\infty:\nonumber\\
&\;\;\;\;
{\mathscr G}\psi\in C_b(\R^{1+N}),\;\,{\rm supp}(\psi)\subset [-M,M]\times\R^N,~\mbox{for some } M>0
\bigg\},
\end{align*}
is a core for the operator $G_p$ for any $p\in [1,+\infty)$.
\end{theorem}

Under stronger assumptions, $C^{\infty}_c(\R^{1+N})$ is a core of $(T(t))$.
More specifically,

\begin{theorem}[{\cite[Theorem 4.1]{lorenzi-zamboni}}]
\label{core non periodic} Let Hypotheses $\ref{hyp1}(ii)$-$(iii)$ be
satisfied. Further, let the coefficients $q_{ij}$ and $b_j$
$(i,j=1,\ldots,N)$ belong to $C^{\alpha/2,2+\alpha}_{\rm
loc}(\R^{1+N})$ and to $C^{\alpha/2,1+\alpha}_{\rm loc}(\R^{1+N})$,
respectively, for some $\alpha\in (0,1)$. Fix $p\in (1,+\infty)$ and
assume that there exist a strictly positive function $V \in
C^2(\R^N)$ blowing up as $|x|\to +\infty$, and a constant $c>0$ such
that the functions
\begin{eqnarray*}
(s,x)\mapsto e^{-c|s|}\frac{({\mathscr A} V)(s,x)}{V(x) \log
V(x)}\qquad (s,x)\mapsto e^{-c|s|}\frac{\langle Q(s,x) \nabla V(x), \nabla V(x)
\rangle}{(V(x))^2 \log V(x)},
\end{eqnarray*}
belong to $L^p(\R^{1+N},\mu)$. Then,
$C_c^{\infty}(\R^{1+N})$ is a core for the operator $G_p$.
\end{theorem}

Sufficient conditions for Theorem \ref{core non periodic} hold are given in terms of the coefficients of the operator ${\mathscr A}$
as follows.
\begin{hyp}
\label{hyp-luca}
\par
\noindent
\begin{enumerate}[\rm (i)]
\item
The coefficients $q_{ij}$ and $b_i$ belong to
$C^{\alpha/2,2+\alpha}_{\rm loc}(\R^{1+N})$ and to
$C^{\alpha/2,1+\alpha}_{\rm loc}(\R^{1+N})$, respectively, for any $i,j=1,\ldots,N$.
Moreover, $q_{ij}=q_{ji}$, for any $i,j=1,\ldots,N$, and there exists a positive constant $\eta_0$ such that
\begin{eqnarray*}
\langle Q(s,x) \xi,\xi \rangle\geq \eta_0 |\xi|^2, \qquad \xi \in
\R^N, \quad (s,x) \in \R^{1+N}.
\end{eqnarray*}
\item
There exists a positive constant $k$ such that
\begin{align*}
&(a)~\sup_{(s,x)\in\R\times B_M}\left (|q_{ij}(s,x)|+e^{-k|s|}|b_j(s,x)|\right )<+\infty,\\
&(b)~\sup_{(s,x)\in\R\times B_M}\langle b(s,x),x\rangle<+\infty,
\end{align*}
for any $M>0$ and any $i,j=1,\ldots,N$.
\item
There exist $\beta,\gamma>0$ such that
\begin{equation*}
\lim_{|x| \to +\infty}\,\sup_{s\in\R}\left(\gamma\Lambda_s(x) |x|^{\beta}+\langle b(s,x),x\rangle\right )=-\infty,
\end{equation*}
where $\Lambda_s(x)$ denotes the maximum eigenvalue of the matrix $Q(s,x)$.
\item
There exists $\delta>0$ such that $\beta\delta<\gamma$,
\begin{eqnarray*}
\limsup_{|x|\to +\infty}\,
\sup_{s\in\R}
\frac{|x|^{\beta-2}\Lambda_s(x)}{\exp{(\delta p^{-1} |x|^{\beta})}\exp(k |s|)} <
+\infty
\end{eqnarray*}
and
\begin{eqnarray*}
\limsup_{|x|\to +\infty}\,\sup_{s\in\R}\frac{| \langle b(s,x),x
\rangle |}{|x|^{2+ \beta(p'-1)} \exp{(\delta (p'-1) |x|^{\beta})}\exp(k |s|)}
< +\infty,
\end{eqnarray*}
where $p'$ is the conjugate index of $p$.
\end{enumerate}
\end{hyp}

\begin{example}
{\rm Let the operator ${\mathscr A}$ be defined by
\begin{eqnarray*}
({\mathscr A}\varphi)(s,x)= (1+|x|^2)^p(\Delta_x \varphi)(s,x) - g(s)(1+|x|^2)^{q} \sum_{j=1}^Nx_jD_j\varphi(x),
\end{eqnarray*}
for any $(s,x)\in\R^{1+N}$,
on smooth functions $\varphi:\R^N\to\R$. Here, $p\in {\mathbb N}\cup\{0\}$, $q\in {\mathbb N}$ satisfy $p< q$.
Further, $g:\R\to\R$ is any function which belongs to
$C^{\alpha}_{\rm loc}(\R)$ for some $\alpha\in (0,1)$ and
satisfies $L^{-1}\le g(s)\le Le^{c|s|}$ for any $s\in\R$ and some
$L>0$. Then, ${\mathscr A}$ satisfies Hypothesis \ref{hyp-luca}.}
\end{example}

\section{The periodic case}
\setcounter{equation}{0}
\label{sect-6}
The case when the coefficients of the operator ${\mathscr A}$ are periodic with respect to $s$ is of particular interest
since in this setting a satisfactory asymptotic analysis of the behaviour of the function $P(t,s)f$ as $|t-s|\to +\infty$ can be
carried over. We address this point in the forthcoming section.
Here, we just list some main differences with respect to the general case dealt with in the previous sections.

We will consider functions defined in $\R^{1+N}$ which are $T$-periodic with respect to the variable $s$,
for some $T>0$. We conveniently identify them with functions defined in ${\mathbb T}\times\R^N$ where ${\mathbb T}=[0,T]$
mod. $T$. We thus denote by $C_b({\mathbb T}\times\R^N)$ (resp. $C^{\alpha/2,\alpha}_{\rm loc}({\mathbb T}\times\R^N)$ $\alpha\in (0,1)$)
the set of functions $f:\R^{1+N}\to\R$ which are bounded, continuous (resp. locally $\alpha$-H\"older continuous with respect
to the parabolic distance of $\R^{1+N}$) and such that $f(s+T,x)=f(s,x)$ for any $(s,x)\in\R^{1+N}$.

If the coefficients of the operator ${\mathscr A}$ satisfy Hypothesis \ref{hyp1} and are $T$-periodic with respect to the variable
$s$, then $P(t+T,r+T)f=P(t,r)f$ for any $r,t\in\R$ with $r<t$.
This property shows that the evolution semigroup $(T(t))$ defined by \eqref{sem-evol} maps $C_b(\mathbb T\times\R^N)$ into itself.
$(T(t))$ is a contractive semigroup in $C_b(\mathbb T\times\R^N)$ but it fails to be strongly continuous.
It is not strong Feller, but it improves spatial regularity.
More precisely, for any $f\in C_b(\mathbb T\times\R^N)$ and any $t>0$, the function $T(t)f$ is twice continuously differentiable
in $\R^{1+N}$ with respect to the spatial variables.

One can define the concept of the weak generator $G_{\infty}^{\sharp}$ of the restriction of $T(t)$ to $C_b(\mathbb T\times\R^N)$, which turns out to be the part of $G_{\infty}$ in $C_b(\mathbb T\times\R^N)$ with
\begin{equation}
D(G_{\infty}^{\sharp})=D(G_{\infty})\cap C_b({\mathbb T}\times\R^N)
\label{dominio-per}
\end{equation}
as a domain, where $D(G_{\infty})$ is given by \eqref{DGinfty}.

\subsection{Invariant measure and periodic evolution system of invariant measures}

In the periodic case, under Condition \eqref{cond-mis-inv} one can
prove the existence of a periodic evolution system of invariant
measures, i.e. an evolution system of invariant measures such that $\mu_{s+T}=\mu_s$ for any
$s\in\R$. As it has been already stressed, evolution systems of invariant measures are, in
general, infinitely many. But only one of them is $T$-periodic.

\begin{theorem}[{\cite[Proposition 2.10]{lorenzi-lunardi-zamboni}}]
Under Hypothesis $\ref{hyp1}$ and assuming that the coefficients of ${\mathscr A}$ are $T$-periodic with respect to the variable $s$,
there exists a unique $T$-periodic evolution system of invariant measure for $(P(t,s))$.
\end{theorem}

Let us denote by $\{\mu_s^{\sharp}: s\in\R\}$ the unique periodic evolution system of measures for the evolution operator $(P(t,s))$. Starting from
this system we define a Borel measure on $(0,T)\times\R^N$ setting
\begin{equation}
\mu^{\sharp}(A\times B)=\frac{1}{T}\int_A\mu_s^{\sharp}(B)ds,
\label{mu-sharp}
\end{equation}
on Borel sets $A\subset (0,T)$ and $B\subset\R^N$, and then extending it to all the Borel set of $(0,T)\times\R^N$.

$\mu^{\sharp}$ is a probability measure and it is invariant for $(T(t))$. Indeed,
\begin{eqnarray*}
\int_{(0,T)\times \R^N}T(t)fd\mu^{\sharp}=\int_{(0,T)\times \R^N}fd\mu^{\sharp},
\end{eqnarray*}
for any $f\in C_b(\mathbb T\times\R^N)$.

Let us denote by $L^p(\mathbb T\times\R^N,\mu^{\sharp})$ the set of all functions $f:\R^{1+N}\to\R$
such that $f(\cdot+T,\cdot)=f$ almost everywhere in $\R^{1+N}$ and
$\int_{(0,T)\times\R^N}|f|^pd\mu^{\sharp}<+\infty$. $L^p(\mathbb T\times\R^N,\mu^{\sharp})$ is a Banach space
when endowed with the norm
\begin{eqnarray*}
\|f\|_{L^p(\mathbb T\times\R^N,\mu^{\sharp})}^p=\int_{(0,T)\times\R^N}|f|^pd\mu^{\sharp},\qquad\;\,
f\in L^p(\mathbb T\times\R^N,\mu^{\sharp}).
\end{eqnarray*}

$(T(t))$ extends to $L^p(\mathbb T\times\R^N,\mu^{\sharp})$ with a strongly continuous semigroup of contractions.
In the case when ${\mathscr A}$ is the nonautonomous $T$-periodic Ornstein-Uhlenbeck operator,
the domain of the infinitesimal generator $G_p^{\sharp}$ of
$(T(t))$ in $L^p(\mathbb T\times\R^N,\mu^{\sharp})$ has been characterized in the case when $p=2$.

\begin{theorem}[{\cite[Theorem 1.2]{geissert-lunardi}}]
Suppose that Condition \eqref{cond-U} is satisfied. Then,
\begin{eqnarray*}
D(G_2^{\sharp})=\{u\in W^{1,2}_{2,\rm loc}(\R^{1+N}): D_tu, D_iu, D_{ij}u\in L^2(\mathbb T\times\R^N,\mu^{\sharp})\}
\end{eqnarray*}
In particular, $D(G_2^{\sharp})$ is compactly embedded in $L^2(\mathbb T\times\R^N,\mu^{\sharp})$.
\end{theorem}

For more general nonautonomous operators with $T$-periodic coefficients with respect to $s$, some suitable cores
have been obtained in \cite{lorenzi-lunardi-zamboni,lorenzi-zamboni}.

\begin{theorem}
\label{thm-cores} Suppose that Hypothesis \ref{hyp1}, Condition
\eqref{cond-U} are satisfied and the coefficients are $T$-periodic
with respect to $s$. Then, the following properties are satisfied.
\begin{enumerate}[\rm (i)]
\item
$D(G_{\infty}^{\sharp})$ $($see \eqref{dominio-per}$)$ is a core of $G_p^{\sharp}$ for any $p\in [1,+\infty)$ $($
\cite[Theorem 6.7]{lorenzi-zamboni}$)$;
\item
for any $\tau \in \R$, $\chi \in
C^{\infty}_{c}(\R^N)$ and $\alpha \in C^1_c(\R)$  with ${\rm supp}(\alpha) \subset (a, a +T)$ for some  $a\geq \tau$, let
$u_{\tau, \chi, \alpha}:\R^{1+N}\to\R$ be the
$T$-periodic $($with respect to $s)$ extension of the function
$(s,x)\mapsto\alpha(s)(P(s,\tau)\chi)(x) $ defined in $[a, a +T)\times
\R^N$. Then, the set ${\mathscr C}=\{u_{\tau,\chi,\alpha}: \tau\in\R,~\alpha\in C^1_c(\R),~\chi\in C^1_c(\R^N)\}$
is a core of $G_p^{\sharp}$ for any $p\in (1,+\infty)$ $($\cite[Proposition 2.12]{lorenzi-lunardi-zamboni}$)$;
\item
suppose that there exists a strictly positive function $V\in C^2(\R^N)$ blowing up as $|x|\to +\infty$, such that
\begin{eqnarray*}
\;\;\;\;\;\;\;\;\;\;\frac{({\mathscr A} V)}{V \log V} \in
L^p(\mathbb T\times\R^N,\mu^{\sharp}) \quad \mbox{and} \quad
\frac{\langle Q\nabla V, \nabla V \rangle}{V^2\log V} \in
L^p(\mathbb T\times\R^N,\mu^{\sharp}),
\end{eqnarray*}
for some $p\in [1,+\infty)$. Then,
\begin{eqnarray*}
\;\;\;\;\;\;\;\;\;\;C_{c}^{\infty}(\mathbb T\times\R^N) := \{ f\in C^{\infty}(\mathbb T\times\R^N):\
{\rm supp}(f)\subset \R\times B_R \mbox{ for some } R>0 \},
\end{eqnarray*}
is a core for the operator $G_p^{\sharp}$ $($\cite[Theorem 6.8]{lorenzi-lunardi-zamboni}$)$.
\end{enumerate}
\end{theorem}

\begin{remark}
\label{stannat} 
{\rm In the case when $p=1$ and under a different set of
assumptions (requiring, in particular, that the diffusion coefficients are bounded),
the result in Theorem \ref{thm-cores}(iii) can be obtained
as a byproduct of the result in \cite[Corollary 1.14]{stannat1}.}
\end{remark}

\begin{example}
{\rm Let the operator ${\mathscr A}$ be defined by
\begin{eqnarray*}
({\mathscr A}\varphi)(s,x)= (1+|x|^2)^p(\Delta_x \varphi)(x) - g(s)(1+|x|^2)^{q}\sum_{j=1}^Nx_jD_j\varphi(x),
\end{eqnarray*}
for any $(s,x)\in\R^{1+N}$, where
$g$ is a positive and $\alpha$-H\"older continuous (for some
$\alpha\in (0,1))$ periodic function, $p\in {\mathbb N}\cup\{0\}$, $q\in {\mathbb N}$
satisfy $p<q$.
Then, ${\mathscr A}$ satisfies the assumptions of Theorem \ref{thm-cores}(iii).}
\end{example}

\section{Asymptotic behaviour}
\setcounter{equation}{0}
\label{sect-7}
In the autonomous case is known that,
whenever an invariant measure exists, it holds that
\begin{equation}
\lim_{t\to +\infty}\|T(t)f-\overline f\|_{L^p(\R^N,\mu)}=0,
\label{convergence-to-zero}
\end{equation}
for any $f\in L^p(\R^N,\mu)$.

In the nonautonomous case, the counterparts of \eqref{convergence-to-zero} are the following formulas
\begin{equation}
\label{CompAs} \lim_{t\to +\infty} \| P(t,s)\varphi -
m_s\varphi\|_{L^p(\R^N, \mu_t)} =0, \quad s\in \R, \; \varphi \in
L^p(\R^N, \mu_s),
\end{equation}
and
\begin{equation}
\label{CompAsInd} \lim_{s\to -\infty} \| P(t,s)\varphi -
m_s\varphi\|_{L^p(\R^N, \mu_t)} =0, \quad t\in \R, \; \varphi \in
C_b(\R^N),
\end{equation}
where
\begin{eqnarray*}
m_s(f)=\int_{\R^N}fd\mu_s,\qquad\;\,s\in\R.
\end{eqnarray*}

In the nonperiodic case, the previous estimates have been proved
in \cite{geissert-lunardi} for the nonautonomous Ornstein-Uhlenbeck operator ${\mathscr A}_O$.
More precisely, Geissert and Lunardi have proved the following result.

\begin{theorem}[{\cite[Proposition 2.17]{geissert-lunardi-2}}]
Let
\begin{eqnarray*}
c_0=\sup\left\{\frac{\kappa_0^2\omega}{M(\omega)^2C^2}: \omega\in (0,\omega_0)\right\},
\end{eqnarray*}
where $\omega_0$ is the supremum of the constant $\omega$ such that \eqref{cond-U} holds true for some $M(\omega)>0$,
$\kappa_0$ is any positive constant such that $\|B(t)x\|\ge\kappa_0\|x\|$ for any $t\in\R$ and any $x\in\R^N$ and $C=\sup_{t\in\R}\|B(t)\|_{\infty}$.
Then,
\begin{equation}
\|P_O(t,s)f-m_s(f)\|_{L^2(\R^N,\mu_t)}\le e^{-c_0(t-s)}\|f\|_{L^2(\R^N,\mu_s)},\qquad\;\,s,t\in\R,\;\,s<t,
\label{estim-GL}
\end{equation}
for any $f\in L^2(\R^N,\mu_s)$.
\end{theorem}

Estimate \eqref{estim-GL} can be extended to any $p\in (1,+\infty)$ by interpolation.
Indeed, since $P_O(t,s)$ is a contraction from $L^1(\R^N,\mu_s)$ into $L^1(\R^N,\mu_t)$ and from $L^{\infty}(\R^N,\mu_s)=L^{\infty}(\R^N)$
into $L^{\infty}(\R^N,\mu_t)=L^{\infty}(\R^N)$ (recall that each measure $\mu_r$ is equivalent to the Lebesgue measure), we can estimate
\begin{align*}
&\|P_O(t,s)f-m_s(f)\|_{L^1(\R^N,\mu_t)}\le 2\|f\|_{L^1(\R^N,\mu_s)},\\[1mm]
&\|P_O(t,s)f-m_s(f)\|_{L^{\infty}(\R^N,\mu_t)}\le 2\|f\|_{L^{\infty}(\R^N,\mu_s)}.
\end{align*}
Stein interpolation theorem now yields
\begin{eqnarray*}
\|P_O(t,s)f-m_s(f)\|_{L^p(\R^N,\mu_t)}\le C_pe^{-c_p(t-s)}\|f\|_{L^p(\R^N,\mu_s)},
\end{eqnarray*}
for any $p\in (1,+\infty)$, where
\begin{eqnarray*}
c_p=\left\{
\begin{array}{ll}
2\left (1-\frac{1}{p}\right )c_o, & p\in (1,2),\\[2mm]
\frac{2}{p}c_0, & p\in [2,+\infty),
\end{array}
\right.
\qquad
C_p=\left\{
\begin{array}{ll}
2^{\frac{2}{p}-1}, & p\in (1,2),\\[1mm]
2^{1-\frac{2}{p}}, & p\in [2,+\infty).
\end{array}
\right.
\end{eqnarray*}

For more general nonautonomous operators, the asymptotic behaviour of $P(t,s)$ is well understood in the
case when coefficients are time-periodic (see \cite{lorenzi-lunardi-zamboni}). Very recently, some of the results in
\cite{lorenzi-lunardi-zamboni} have been proved in the nonperiodic case (see
\cite{angiuli-lorenzi-lunardi}) when the diffusion coefficients are bounded and independent of the spatial variables.
The general nonperiodic case is still under investigation.

\subsection{The periodic case}
The key tool to prove Estimates \eqref{CompAs} and \eqref{CompAsInd}
is the analysis of the asymptotic behavior of the evolution
semigroup $(T(t))$ in the spaces $L^p(\R^{1+N},\mu^{\sharp})$,
where, we recall that the measure $\mu^{\sharp}$ is the only
probability measure which extends the function in \eqref{mu-sharp} to
the $\sigma$ algebra of all the Borel sets of $(0,T)\times\R^N$, and
$\{\mu^{\sharp}_s: s\in\R\}$ is the unique $T$-periodic evolution systems of invariant measures of
$(P(t,s))$.

We stress that the classical arguments for evolution semigroups (see e.g., the monograph \cite{chicone-latushkin}) cannot be applied
to study the long time behaviour of the function $P(t,s)f-m_sf$ in the $L^p$-spaces associated with the evolution system of invariant measures $\{\mu_s^{\sharp}: s\in\R\}$.
Indeed, the classical theory requires that $T(t)$ maps $L^p(\mathbb T;X)$ into itself, which of course is not the case
since $P(t,s)$ maps $L^p(\R^N,\mu_s)$ into $L^p(\R^N,\mu_t)$ and these $L^p$-spaces differ, in general.
Nevertheless, there is still a link between \eqref{CompAs},  \eqref{CompAsInd} and the asymptotic behaviour of the evolution semigroup $(T(t))$. This link  is made clear by
the following theorem.
\begin{theorem}[{\cite[Theorem 3.1]{lorenzi-lunardi-zamboni}}]
Suppose that Hypotheses $\ref{hyp1}(i)$-$(ii)$ and \eqref{cond-mis-inv} are satisfied.
For $1\leq p < +\infty$, consider the following statements:
\begin{enumerate}[\rm (i)]
\item
for any $f\in  L^p(\mathbb T \times \R^N, \mu^{\sharp})$ we have
\begin{equation}
\label{CompAsT(t)}
\lim_{t\to +\infty} \|T(t)(f-\Pi f) \|_{L^p(\mathbb T \times \R^N, \mu^{\sharp})} =0;
\end{equation}
\item
for any $ \varphi  \in C_b(\R^N)$ we have
\begin{eqnarray*}
\exists/\forall t\in \R, \quad  \lim_{s\to -
\infty} \| P(t,s)\varphi - m_s\varphi\|_{L^p(\R^N, \mu_t^{\sharp})} =0;
\end{eqnarray*}
\item
for some/any  $s\in \R$   we have
\begin{eqnarray*}
\lim_{t\to +\infty} \| P(t,s)\varphi -
m_s\varphi\|_{L^p(\R^N, \mu_t^{\sharp})} =0, \quad  \varphi  \in L^p(\R^N,
\mu_s^{\sharp});
\end{eqnarray*}
\item
for  any $ \varphi  \in C_b(\R^N)$   we have
\begin{eqnarray*}
\exists /\forall t \in \R, \quad \lim_{s\to
- \infty} \| P(t,s)\varphi - m_s\varphi\|_{L^{\infty}(B_R)} =0,
\quad   R>0;
\end{eqnarray*}
\item
for some/any $s\in \R$   we have
\begin{eqnarray*}
\lim_{t\to +\infty} \| P(t,s)\varphi -
m_s\varphi\|_{L^{\infty}(B_R)} =0, \quad     \varphi  \in
C_b(\R^N), \;R>0.
\end{eqnarray*}
\end{enumerate}
For every $p\in [1,+\infty)$, statements $(i)$, $(ii)$, $(iii)$ are equivalent, and they are implied by statements $(iv)$ and $(v)$.
If in addition Hypothesis $\ref{hyp10}$ holds, for every $p\in [1,+\infty)$ statements $(i)$ to $(v)$ are equivalent.
\end{theorem}

Here, $\Pi$ is the projection on $L^p(\R^N,\mu^{\sharp})$ defined by $(\Pi f)(s,x)=m_s(f)$ for any $(s,x)\in\R^{1+N}$.
Note that $\Pi$ commutes with the semigroup $(T(t))$.

\begin{remark}
{\rm The convergence of $P(t,s)\varphi-m_s\varphi$ to zero is not uniform in
$\R^N$, in general, for $\varphi\in C_b(\R^N)$.
Take for instance any Ornstein-Uhlenbeck operator
\begin{eqnarray*}
({\mathscr A}_O\varphi)(x) =\sum_{i,j=1}^Nq_{ij}D_{ij}\varphi(x) + \sum_{i,j=1}^Nb_{ij}x_jD_i\varphi(x),
\end{eqnarray*}
where $Q $ is  symmetric and positive definite   and all the
eigenvalues of  $B $ have negative real part. Then, $P_O(t,s)=T(t-s)$ and
$\mu_t=\mu$ where $\mu$ is the invariant measure of the associated autonomous
Ornstein-Uhlenbeck operator.
Let $f=e^{i\langle\cdot,h\rangle}$ for some $h\in\R^N\setminus\{0\}$. Then,
\begin{eqnarray*}
P_O(t,s)f = \exp\left (-\langle Q_{t-s}h,h\rangle
+i\langle \cdot,e^{(t-s)B^*}h\rangle\right ),\qquad\;\,s<t,
\end{eqnarray*}
where $Q_{r} :=  \int_0^{r}
e^{\sigma B}Qe^{\sigma B^*}d\sigma$.
Since
$m_s(f) = e^{-\langle Q_{\infty} h,h\rangle}$, it holds that
\begin{align*}
P_O(t,s)f - m_s(f) =&
\left \{ \exp\left (-\langle Q_{t-s}  h,h\rangle\right )
- \exp\left (-\langle Q_{\infty}  h,h\rangle\right )\right\}e^{i\langle \cdot,e^{(t-s)B^*}h\rangle}\\
&+ \exp\left (-\langle
Q_{\infty} h,h\rangle\right ) \left (\exp( i\langle
\cdot,e^{(t-s)B^*}h\rangle)-1\right ),
\end{align*}
for any $t>s$. Note that
\begin{eqnarray*}
\sup_{x\in\R^N}\left |\exp( i\langle
\cdot,e^{(t-s)B^*}h\rangle)-1\right |=2.
\end{eqnarray*}
Hence,
\begin{align*}
\lim_{t\to +\infty}\|P_O(t,s)f - m_s(f)\|_{\infty} = \lim_{s\to -\infty}\|P_O(t,s)f - m_s(f)\|_{\infty}=2\exp\left (-\langle
Q_{\infty} h,h\rangle\right ).
\end{align*}
}
\end{remark}

\begin{remark}
{\rm Let us consider the formula
\begin{equation}
\lim_{t\to +\infty} \| P(t,s)\varphi - m_s\varphi\|_{L^p(\R^N,
\mu_t^{\sharp})} =0, \quad  \varphi  \in L^p(\R^N, \mu_s^{\sharp}).
\label{conv-to-zero-L^p}
\end{equation}
Here, $t$ appears both in the evolution operator and in the measure
$\mu_t$, so that one might wonder that the convergence to zero of
$\| P(t,s)\varphi - m_s\varphi\|_{L^p(\R^N, \mu_t^{\sharp})}$ is due
to the convergence to zero of the density $\varrho^{\sharp}$ of the
measure $\mu^{\sharp}$ as $t\to +\infty$, this making Formula
\eqref{conv-to-zero-L^p} somehow trivial. (Note that the density of
$\mu_t^{\sharp}$ is the function $\varrho^{\sharp}(t,\cdot)$ by the
disintegration theorem for measures.) But this is not the case.
Indeed, since $\{\mu_t^{\sharp}: t\in\R\}$ is a $T$-periodic
evolution systems of measures, it turns out that
$\varrho^{\sharp}(t+T,\cdot)=\varrho^{\sharp}(t,\cdot)$ for any $t$,
so that $\varrho^{\sharp}(t,\cdot)$ cannot vanish as $t\to +\infty$.}
\end{remark}

Also the exponential convergence to zero of
$\|P(t,s)f-m_s(f)\|_{L^p(\R^N,\mu_t^{\sharp})}$ as $t-s\to +\infty$ can be related to the exponential convergence to
zero of the function $T(t)(f-\Pi f)$ as the following theorem shows.

\begin{theorem}[{\cite[Theorem 3.2]{lorenzi-lunardi-zamboni}}]
Let Hypothesis $\ref{hyp1}$ hold. Fix $1\leq p
\leq +\infty$, $M>0$, $\omega \in \R$. The following conditions are
equivalent:
 \begin{enumerate}[\rm (a)]
 \item
 for every $t>0$ and $u\in  L^p(\mathbb T \times \R^N,\mu^{\sharp})$,
\begin{eqnarray*}
\;\;\|T(t)(I-\Pi)u\|_{L^p(\mathbb T\times\R^N,\mu^{\sharp})}\leq Me^{\omega t} \|u\|_{L^p(\mathbb T\times\R^N,\mu^{\sharp})};
\end{eqnarray*}
 \item
for every $t>s$ and $ \varphi \in L^p(\R^N, \mu_s^{\sharp})$,
\begin{eqnarray*}
\;\;\;\| P(t,s)\varphi  - m_s(\varphi)\|_{L^p(\R^N, \mu_t^{\sharp})} \leq
Me^{\omega(t-s)}\|\varphi \|_{L^p(\R^N, \mu_s^{\sharp})}.
\label{2}
\end{eqnarray*}
 \end{enumerate}
\end{theorem}

Since $T(t)$ commutes with $\Pi$ for any $t>0$, $(T(t)(I-\Pi))$ is nothing but the part of $(T(t))$ in $(I-\Pi)(L^p(\mathbb T\times\R^N,\mu^{\sharp}))$.
Hence, $T(t)(I-\Pi)$ converges to zero with exponential rate if and only if the
growth bound of the semigroup $(T(t)(I-\Pi))$ is negative or, equivalently,
if the spectral bound of $G_p^{\sharp}$ is negative, since $(T(t))$ and its part in $(I-\Pi)(L^p(\mathbb T\times\R^N,\mu^{\sharp}))$ satisfy the
spectral mapping theorem (see \cite[Theorem 2.17 \& 3.15]{lorenzi-lunardi-zamboni}). Computing explicitly the spectrum/growth bound
is an hard task in general. It has been computed in the case of the nonautonomous
Ornstein-Uhlenbeck operator.

\begin{theorem}[{\cite[Corollary 2.11]{geissert-lunardi-2}}]
The growth bound of the part of $(T(t))$ in $(I-\Pi)(L^2(\mathbb T\times\R^N,\mu^{\sharp}))$ is
$\omega_0$ $($where, we recall,
$\omega_0$ is the supremum of the constant $\omega$ such that \eqref{cond-U}$)$.
\end{theorem}

For more general nonautonomous operators one can prove the following.

\begin{theorem}[{\cite[Theorem 3.6]{lorenzi-lunardi-zamboni}}]
Let  Hypotheses  $\ref{hyp1}$  and  $\ref{hyp10}$ hold.
Set
\begin{equation}
\label{omegap,gammap}
 \omega_{p}: =\inf A_p, \quad  \gamma_{p} :=\inf B_p,
\end{equation}
where
\begin{align*}
A_p : = \{ & \omega\in\R: \exists M_{\omega}>0 \;\;\mbox{s.t.}\\[1mm]
&\|T(t)(f-\Pi f)\|_{L^p(\mathbb T\times\R^N,\mu^{\sharp})}\le M_{\omega} e^{\omega t}\|f-\Pi f\|_{L^p(\mathbb T\times\R^N,\mu^{\sharp})}
\notag
\\
&\forall t\geq 0,\, f\in L^p(\mathbb T \times \R^N,\mu^{\sharp})\},\\[1mm]
B_p := \{ & \omega\in\R: \exists
N_{\omega}>0 \;\;\mbox{s.t.}\;\,\|\,|\nabla_x T(t)f|\, \|_{L^p(\mathbb T\times\R^N,\mu^{\sharp})}\le
N_{\omega} e^{\omega t}\|f\|_{L^p(\mathbb T\times\R^N,\mu^{\sharp})}\notag
\\
&\forall t\geq 1,\, f\in L^p(\mathbb T\times \R^N,\mu^{\sharp})\}.
\end{align*}
Then  $A_p\subset B_p$ for every $p\in (1,+\infty)$ such that
$\ell_p<+\infty$ $($see \eqref{ellp}$)$. If the diffusion coefficients are bounded, $B_p \subset A_p $ for every $p\geq 2$.
\end{theorem}

We recall that, if Hypotheses \ref{hyp1} and \ref{hyp10} are satisfied
then
\begin{eqnarray*}
|\nabla_x P(t,s)\varphi(x)|^p \leq C^p \max\{(t-s)^{-p/2}, \,1\}
e^{p\ell_p(t-s)}P(t,s)|\varphi |^p(x),\qquad\;\,x\in\R^N,
\end{eqnarray*}
for any $\varphi\in C_b(\R^N)$. Hence, for any $f\in C_b(\R^{1+N})$ it follows that
\begin{align*}
\int_{(0,T)\times\R^N}|\nabla_xT(t)f|^pd\mu^{\sharp}=&
\frac{1}{T}\int_0^Tds\int_{\R^N}|\nabla_x P(s,s-t)f(s-t,\cdot)|^pd\mu_s^{\sharp}\\
\le&\frac{C^p}{T}e^{p\ell_pt}\int_0^Tds\int_{\R^N}P(s,s-t)|f(s-t,\cdot)|^pd\mu_s^{\sharp}\\
=&\frac{C^p}{T}e^{p\ell_pt}\int_0^Tds\int_{\R^N}|f(s-t,\cdot)|^pd\mu_{s-t}^{\sharp}\\
=&\frac{C^p}{T}e^{p\ell_pt}\int_0^Tds\int_{\R^N}|f(s,\cdot)|^pd\mu_s^{\sharp},
\end{align*}
for any $t\ge 1$. Hence,
\begin{eqnarray*}
\|\nabla_xT(t)f\|_{L^p((0,T)\times\R^N,\mu^{\sharp})}\le Ce^{\ell_pt}\|f\|_{L^p((0,T)\times\R^N,\mu^{\sharp})},\qquad\;\,t\ge 1.
\end{eqnarray*}
Clearly, this inequality can be extended to any $f\in L^p(\mathbb T\times\R^N,\mu^{\sharp})$ by density. This shows that
$\ell_p\in B_p$. Hence, a sufficient condition guaranteeing that
$\|P(t,s)f-m_s(f)\|_{L^p(\R^N,\mu_t^{\sharp})}$ decreases to zero as $t-s\to +\infty$ with exponential rate is that
$\ell_p<0$.

Even without the assumptions $\ell_p<0$ we can prove that the
function $\|T(t)(I-\Pi)f\|_{L^p({\mathbb
T}\times\R^N,\mu^{\sharp})}$ tends to $0$ as $t\to +\infty$. More
precisely, the following result holds true.

\begin{theorem}[Theorem 3.5 of \cite{lorenzi-lunardi-zamboni}]
\label{thm-3.5}
Let Hypotheses $\ref{hyp1}$ and \eqref{hyp10} be
satisfied. Further assume either that the diffusion coefficients of
the operator ${\mathscr A}$ are bounded or there exists a positive
constant $C$ such that
\begin{eqnarray*}
\| Q(s,x)\|_{L(\R^N)} \leq C(|x|+1) V(x), \;\; \langle b(s,x), x\rangle \leq C(|x|^2+1)V(x),
\end{eqnarray*}
for any $(s,x)\in\R^{1+N}$.
Then, for every $p\in [1,+\infty)$ Estimate \eqref{CompAsT(t)} holds
true.
\end{theorem}

\begin{proof}
We sketch the proof since it can be applied also to the autonomous setting.

Some reductions are in order.
Of course, it is enough to prove \eqref{CompAsT(t)} in the case when $p=2$. Indeed, the general case when $p\neq 2$ then follows
by applying Stein interpolation theorem, since $(T(t))$ is bounded both in $L^1(\mathbb T\times\R^N,\mu^{\sharp})$
and in $L^{\infty}(\mathbb R\times\R^N,\mu^{\sharp})=L^{\infty}(\mathbb T\times\R^N)$. Moreover, it is enough to prove
\eqref{CompAsT(t)} for functions $f$ in the core ${\mathscr C}$ (see Theorem \ref{thm-cores}(ii)).

The proof consists of three steps.
\begin{description}
\item[{\it Step 1}] One shows that $\nabla_xT(t)f$ tends to $0$ as $t\to +\infty$ in $L^2(\mathbb T\times\R^N,\mu^{\sharp})$.
\item[{\it Step 2}] One proves that, from any sequence $(t_n)$
diverging to $+\infty$, one can extract a subsequence $(t_{n_k})$
such that $T(t_{n_k})(I-\Pi)f$ converges in $L^2(\mathbb
T\times\R^N,\mu^{\sharp})$ to some function $g\in L^2(\mathbb
T\times\R^N,\mu^{\sharp})$ as $k\to +\infty$.
\item[{\it Step 3}] Using Steps 1 and 2 one concludes that $g\equiv 0$.
\end{description}

To prove the convergence to zero of $\nabla_xT(t)f$ one takes advantage of the formula
\begin{equation}
 \int_{(0,T) \times \R^N }  \langle Q \nabla_x f, \nabla_x f
\rangle\,d \mu^{\sharp} \leq  - \int_{\mathbb T \times \R^N } f  G_p
f \, d \mu^{\sharp};
\label{carreduchamps}
\end{equation}
which holds true under our assumptions and is, in fact, an equality
if the diffusion coefficients are bounded. In this latter case
Formula \eqref{carreduchamps} is the so called \emph{identit\'e de
carr\'e du champ}.

Formula \eqref{carreduchamps} can be proved heuristically observing
that
\begin{equation}
\int_{(0,T)\times\R^N}{\mathscr G}u\,d\mu^{\sharp}=0,\qquad\;\,u\in
D(G_p). \label{infinitinvper}
\end{equation}
If we formally insert $u=f^2$ in \eqref{infinitinvper} and notice
that
\begin{eqnarray*}
{\mathscr G}(f^2)=2f{\mathscr G}f+\langle
Q\nabla_xf,\nabla_xf\rangle,
\end{eqnarray*}
we immediately end up with the identit\'e de carr\'e du champ
\begin{equation}
\int_{(0,T) \times \R^N }  \langle Q \nabla_x f, \nabla_x f
\rangle\,d \mu^{\sharp} =  - \int_{(0,T)\times \R^N } f\,  G_p f \,
d \mu^{\sharp}; \label{carreduchamps-1}
\end{equation}
The main issue is to make the previous argument rigorous. This is
easy in the case when the diffusion coefficients of the operator
${\mathscr A}$ are bounded. Indeed, in this case, the function $f^2$ is in
$D(G_{\infty}^{\sharp})$. Clearly $f^2$ is bounded and it belongs to
$W^{1,2}_{p,{\rm loc}}(\R^{1+N})$ for any $p<+\infty$. Moreover,
$\nabla_x f$ is bounded and continuous in $\R^{1+N}$, since the
evolution operator satisfies uniform gradient estimates. Hence, the
function ${\mathscr G}(f^2)$ is in $C_b(\mathbb T\times\R^N)$, thus
implying that it belongs to $D(G_{\infty}^{\sharp})\subset D(G_p^{\sharp})$.
The case when the diffusion coefficients are unbounded is a bit
trickier. Indeed, it is not clear if the function $\langle
Q\nabla_xf,\nabla_xf\rangle$ is bounded in $\R^{1+N}$. To overcome
such a difficult, one approximate the function $f^2$ by a sequence
of functions compactly supported in $x$. Taking the limit as $n\to
+\infty$ one ends up with formula \eqref{carreduchamps}.

Using inequality \eqref{carreduchamps} one then proves that
\begin{align*}
\|T(t)f\|_{L^2({\mathbb
T}\times\R^N,\mu^{\sharp})}^2-\|f\|_{L^2({\mathbb
T}\times\R^N,\mu^{\sharp})}^2
=&\int_0^t\frac{d}{dr}\|T(r)f\|_{L^2({\mathbb T}\times\R^N,\mu^{\sharp})}^2dr\\
=&\frac{2}{T}\int_0^tdr\int_{(0,T)\times\R^N}\langle T(r)f,T(r)G_2f\rangle d\mu^{\sharp}\\
\le &
-\frac{2}{T}\int_0^tdr\int_{(0,T)\times\R^N}|\nabla_xT(t)f|^2d\mu^{\sharp},
\end{align*}
for any $t>0$, from which we immediately get
\begin{eqnarray*}
\frac{2}{T}\int_0^tdr\int_{(0,T)\times\R^N}|\nabla_xT(t)f|^2d\mu^{\sharp}\le
\|f\|_{L^2({\mathbb T}\times\R^N,\mu^{\sharp})}^2,\qquad\;\,t>0,
\end{eqnarray*}
or, equivalently, that the function
\begin{eqnarray*}
\chi_f(t)=\int_{\mathbb
T\times\R^N}|\nabla_xT(r)f|^2d\mu^{\sharp},\qquad\;\,t>0
\end{eqnarray*}
is in $L^1((0,+\infty))$. The function $\chi_f$ is differentiable in $(0,+\infty)$ and
using the H\"older inequality, one can easily show that
\begin{eqnarray*}
|\chi_f'(t)|\le 2(\chi_f(t))^{1/2}\chi_{G_2^{\sharp}f}(t))^{1/2}\le \chi_f(t)+\chi_{G_2^{\sharp}f}(t),\qquad\;\,t>0.
\end{eqnarray*}
The same argument as above applied to the function $G_2^{\sharp}f$ shows that $\chi_{G_2^{\sharp}f}$
is in $L^1((0,+\infty))$. Hence, $\chi_f\in W^{1,1}((0,+\infty))$ and this implies that
$\chi_f$ tends to $0$ as $t\to +\infty$.

To prove that $T(t)f$ converges to $0$ as $t\to +\infty$ in
$L^2(\mathbb T\times\R^N,\mu)$ one can employ a compactness
argument. More specifically, let $f$ be the $T$-periodic (with
respect to $s$) extension of the function
$(s,x)\mapsto\alpha(s)(P(s,\tau)\chi)(x) $ defined in $[a, a
+T)\times\R^N$, where $\alpha$ and $\chi$ are compactly supported
functions with ${\rm supp}\alpha\subset (a, a+T)$. Hence,
$T(t)(f-\Pi f)$ is the $T$-periodic extension of the function
$(s,x)\mapsto \alpha(s-t)((P(s,\tau)\chi)(x)-m_{\tau}\chi)$. One
proves that the set $\{T(t)(I-\Pi)f: t>0\}$ is equibounded (this is
clear) and equicontinuous (this is a bit trickier). By
Arzel\`a-Ascoli theorem, there exists a sequence $(t_n)$ diverging
to $+\infty$ such that $T(t_n)(I-\Pi)f$ converges to a function
$g\in C_b(\mathbb T\times\R^N)$ locally uniformly in $\mathbb
T\times\R^N$. As a by product, $T(t_n)(I-\Pi)f$ converges to $g$ in
$L^2(\mathbb T\times\R^N,\mu^{\sharp})$ as $n\to +\infty$.

Next one shows that $g=0$ observing that $g\in (I-\Pi)(L^2(\mathbb
T\times\R^N,\mu^{\sharp}))$ and $\nabla_xg=0$ since $\nabla
T(t_n)(I-\Pi)f$ tends to $0$ in $L^2(\R^{1+N},\mu^{\sharp})$ as
$n\to +\infty$. (Note that $(I-\Pi)(L^2(\mathbb
T\times\R^N,\mu^{\sharp}))$ can be identified with $L^2({\mathbb
T})$.)
\end{proof}

\begin{remark}
{\rm We mention that the \emph{identit\'e de carr\'e du champ}
\eqref{carreduchamps-1} has been proved for $p> 2$ also in some
situation where the diffusion coefficients are unbounded. This is the case when
\begin{eqnarray*}
\| Q(s,x)\|_{L(\R^N)} \leq C(|x|+1) V(x), \;\; |\langle
b(s,x), x\rangle | \leq C(|x|^2+1) V(x),
\end{eqnarray*}
for any $(s,x)\in \R^{1+N}$ and some positive constant $C$. The \emph{identit\'e de carr\'e du
champ} reads:
\begin{eqnarray*}
 \int_{(0,T)\times \R^N}  |u|^{p-2} \langle Q \nabla_x u, \nabla_x u
\rangle\chi_{\{u\neq 0\}} \,d \mu^{\sharp} = - \frac{1}{p-1}
\int_{(0,T) \times \R^N} u |u|^{p-2}G_p^{\sharp} u \, d \mu^{\sharp},
 \end{eqnarray*}
for any $u\in D(G_{\infty}^{\sharp})$.

Similarly, under the Hypothesis of Theorem \ref{thm-3.5}, the
inequality
\begin{eqnarray*}
\int_{(0,T)\times \R^N}  |u|^{p-2} \langle Q \nabla_x u, \nabla_x u
\rangle\chi_{\{u\neq 0\}} \,d \mu^{\sharp} \leq - \frac{1}{p-1}
\int_{(0,T)\times \R^N} u |u|^{p-2}G_p^{\sharp} u \, d \mu^{\sharp},
\end{eqnarray*}
holds true for any $p\in (1,+\infty)$ and any $u\in D(G_{\infty}^{\sharp})$
(see \cite[Proposition 2.15]{lorenzi-lunardi-zamboni}).

In particular, this latter inequality allows to show that $D(G_p^{\sharp})$
is continuously embedded into $W^{0,1}_p({\mathbb T}\times\R^N,
\mu^{\sharp})$.
More precisely, it allows to show that the mapping $f\mapsto
Q^{1/2}\nabla_xf$ is bounded from $D(G_p^{\sharp})$ into $(L^p(\mathbb
T\times\R^N,\mu^{\sharp}))^N$.}
\end{remark}

\section{Some insight on the spectrum of $G_p^{\sharp}$}
\setcounter{equation}{0}
\label{sect-8}
Even if the spectrum of the generator $G_p^{\sharp}$ of the semigroup $(T(t))$ is not explicitly known some remarkable
results are available.

\begin{theorem}[{\cite[Theorems 3.15 \& 3.16]{lorenzi-lunardi-zamboni}}]
Let Hypotheses  $\ref{hyp1}$  and  $\ref{hyp10}$ hold.
Further, assume that the diffusion coefficients are independent of $x$ and the supremum of the function
$r$ in Hypothesis $\ref{hyp10}$ is negative. Then, for any $p\in
(1,+\infty)$, $D(G_p^{\sharp})$ is compactly embedded in $L^p(\mathbb T \times
\R^N,\mu)$.
Moreover,
\begin{enumerate}[\rm (i)]
\item
the spectrum of $G_p^{\sharp}$ consists of isolated eigenvalues independent
of $p$, for $p\in (1,+\infty)$. The associated spectral projections are independent of $p$, too;
\item
the growth bounds $\omega_p $ defined in \eqref{omegap,gammap} are
independent of $p\in (1,+\infty)$. Denoting by $\omega_0$ their
common value, for every $p\in (1,+\infty)$ we have
\begin{eqnarray*}
\omega_0=  \sup \,\{{\rm Re}\,\lambda: \lambda \in
\sigma(G_p^{\sharp})\setminus i\R \}.
\end{eqnarray*}
\end{enumerate}
\end{theorem}

As in the autonomous case, the main tool in the proof of the
previous theorem is the Log-Sobolev inequality, which reads:
\begin{eqnarray*}
 \label{logsob ineq mod}
\int_{(0,T) \times \R^N } |u|^2 \log (|u|) d \mu^{\sharp}  \leq
\frac{1}{2}\int_0^T \Pi |u|^2 \log (\Pi |u|^2)ds + \frac{
\Lambda}{|r_0|} \int_{(0,T) \times \R^N } |\nabla_x u|^2 d
\mu^{\sharp},
\end{eqnarray*}
for any $u\in D(G_{\infty}^{\sharp})$,
where $\Lambda$ denotes the supremum of the eigenvalues of the matrix $Q(s)$ when $s$ varies in $[0,T]$.


The results in the previous theorem apply, in particular, in the case of the periodic nonautonomous Ornstein-Uhlenbeck
operator. In this situation, as in the autonomous case, some information on the eigenfunctions
of the operator $G_p^{\sharp}$ is available. More precisely,

\begin{theorem}
Let $\lambda$ be an eigenvalue of the operator $G_p^{\sharp}$ and let $u$ be a corresponding eigenfunction.
Then,
\begin{eqnarray*}
u(s,x)=\sum_{|\alpha|\le K}c_{\alpha}(t)x^{\alpha},\qquad\;\,s\in\R,\;\,x\in\R^N,
\end{eqnarray*}
where $c_{\alpha}\in W^{1,p}(\mathbb T)$ for any $\alpha$ and $K\le \omega_0^{-1}|{\rm Re}\lambda|$,
$\omega_0$ being the supremum of $\omega>0$ such that \eqref{cond-U} holds true for some $C=C(\omega)>0$.
\end{theorem}

\begin{proof}
A proof has been given in \cite[Proposition 2.5]{geissert-lunardi-2} in the case $p=2$ but it can be extended with the same technique
to the case $p\neq 2$. It is obtained adapting the techniques of the autonomous case (see \cite[Proposition 3.2]{metafune-pallara-enrico}) and is based on the pointwise gradient estimates.

Since it is quite easy and show once more the role played by the gradient estimates, we go into details.

Let $\lambda$ be an eigenvalue of $G_p^{\sharp}$ and let $u$ be a corresponding eigenfunction. Then,
$u\in C^{\infty}(\mathbb T\times\R^N)\cap L^2(\mathbb T\times\R^N,\mu^{\sharp})$ and
all its derivatives are in $L^2(\mathbb T\times\R^N,\mu^{\sharp})$ as well.
Moreover, $T_O(t)u=e^{\lambda t}u$ for any $t>0$.
Hence, $D^{\alpha}_xT_O(t)u=e^{\lambda t}D^{\alpha}_x u$ for any multiindex $\alpha$.

Using the gradient estimate
\begin{eqnarray*}
\|D^{\alpha}_xP_O(t,s)f\|_{L^p(\R^N,\mu_t)}\le Ce^{-\omega |\alpha|(t-s)}\|f\|_{L^p(\R^N,\mu_s)},\;\;\;t-s\gg 1,
\;\,f\in L^p(\R^N,\mu_s),
\end{eqnarray*}
proved in \cite[Lemma 3.3]{geissert-lunardi} for any $p\in [1,+\infty)$, any
multiindex $\alpha$, any $\omega<\omega_0$ and some positive
constant $C=C(\omega,\alpha)$, one can easily show that
\begin{eqnarray*}
\|D^{\alpha}_xT_O(t)u\|_{L^p(\mathbb T\times\R^N,\mu^{\sharp})}\le Ce^{-\omega |\alpha|t}\|u\|_{L^p(\mathbb T\times\R^N,\mu^{\sharp})},\quad\;\,t\gg 1.
\end{eqnarray*}
It follows that
\begin{eqnarray*}
e^{\lambda t}\|D^{\alpha}_xu\|_{L^p(\mathbb T\times\R^N,\mu^{\sharp})}\le Ce^{-\omega|\alpha| t}\|u\|_{L^p(\mathbb T\times\R^N,\mu^{\sharp})},
\end{eqnarray*}
from which we can infer that $D^{\alpha}_xu\equiv 0$ if ${\rm Re}\lambda> -\omega|\alpha|$, letting $t\to +\infty$.
Hence $u$ is a polynomial with degree not grater than $\omega^{-1}|{\rm Re}\lambda|$ for any $\omega\in (0,\omega_0)$.
The proof is now complete.
\end{proof}

\subsection*{Acknowledgments} The author wishes to thank the anonymous referee for the careful reading of the paper.


\begin{thebibliography}{99}

\bibitem{angiuli-lorenzi-lunardi}
\newblock
L. Angiuli, L. Lorenzi, and A. Lunardi,
\newblock
\emph{Hypercontractivity and asymptotic behaviour in nonautonomous Kolmogorov equations}.
\newblock
Available on arXiv (http://arxiv.org/abs/1203.1280).

\bibitem{albanese-mangino}
\newblock
A.A. Albanese, and E.M. Mangino,
\newblock
\emph{Cores for Feller semigroups with an invariant measure},
\newblock
J. Differential Equations, \textbf{225} (2006), 361--377.

\bibitem{albanese-mangino-2}
\newblock
A.A. Albanese, and E.M. Mangino,
\newblock
\emph{Corrigendum to: ``Cores for Feller semigroups with an invariant measure''},
\newblock
J. Differential Equations  \textbf{244}  (2008), 2980--2982.

\bibitem{albanese-mangino-lorenzi}
\newblock
A.A. Albanese, L. Lorenzi, and E.M. Mangino,
\newblock
\emph{$L^p$-uniqueness for elliptic operators with unbounded coefficients in $\Bbb R^N$},
\newblock
J. Funct. Anal.,  \textbf{256}  (2009), 1238--1257.

\bibitem{AS}
\newblock
D.G. Aronson, and J. Serrin,
\newblock
\emph{Local behaviour of solutions of quasilinear parabolic equations},
\newblock
Arch. Rational. Mech. Anal. \textbf{25} (1967), 81--122.


\bibitem{azencott}
\newblock
R. Azencott,
\newblock
\emph{Behaviour of diffusion semigroups at infinity},
\newblock
Bull. Soc. Math. France, \textbf{102} (1974), 193--240.

\bibitem{B0}
\newblock
S. Bernstein,
\emph{Sur la g\'en\'eralisation du probl\'eme de Dirichlet, I},
\newblock
Math. Ann., \textbf{62} (1906), 253--271.

\bibitem{BL}
\newblock
M. Bertoldi, and L. Lorenzi,
\newblock
\emph{Estimates of the derivatives for parabolic operators with unbounded coefficients},
\newblock
Trans. Amer. Math. Soc., \textbf{357} (2005), no. 7, 2627--2664.

\bibitem{bertoldi-lorenzi}
\newblock
M. Bertoldi, and L. Lorenzi,
\newblock
``Analytical methods for Markov semigroups,''
\newblock
Vol. \textbf{283} of Pure and applied mathematics,
\newblock
Chapman Hall/CRC Press (2006).

\bibitem{BDPR08}
V.I. Bogachev, G. Da Prato, and M. R\"ockner,
\newblock
\emph{On parabolic equations for measures},
\newblock
Comm. Partial Differential equations \textbf{33} (2008), 397--418.

\bibitem{BDPRS}
\newblock
V.I. Bogachev, G. Da Prato, M. R\"ockner, and W. Stannat,
\newblock
\emph{Uniqueness of solutions to weak parabolic equations for
measures},
\newblock
Bull. Lond. Math. Soc. \textbf{39} (2007), 631--640.

\bibitem{BKR01}
\newblock
V.I. Bogachev, N.V. Krylov, and M. R\"ockner,
\emph{On regularity of transition probabilities and invariant measures of singular
diffusion under minimal conditions},
\newblock
Comm. Partial Differential equations \textbf{26} (2001), 2037--2080.

\bibitem{BRS06}
\newblock
V.I. Bogachev, M. R\"ockner, and S.V. Shaposhnikov,
\newblock
\emph{Global regularity and bounds for solutions of parabolic
equations for probability measures},
\newblock
Theory Probab. Appl. \textbf{50} (2006), 561--581.

\bibitem{BRS08}
\newblock
V.I. Bogachev, M. R\"ockner, and S.V. Shaposhnikov,
\newblock
\emph{Estimates of densities of stationary distributions and
transition probabilities of diffusion processes},
\newblock
Theory Probab. Appl. \textbf{52} (2008), 209-236.

\bibitem{chicone-latushkin}
\newblock
C. Chicone, and Y. Latushkin,
\newblock
``Evolution Semigroups in  Dynamical Systems and Differential Equations,''
\newblock
Amer. Math. Soc., Providence (RI), 1999.



\bibitem{CFMP}
\newblock
R. Chill, E. Fasangova, G. Metafune, and D. Pallara,
\newblock
\emph{The sector of analyticity of the Ornstein-Uhlenbeck semigroup on $L^p$ spaces with respect to invariant measure},
\newblock
J. London Math. Soc., \textbf{71} (2005), no. 3, 703--722.

\bibitem{daprato-lunardi}
\newblock
G. Da Prato, and A. Lunardi,
\newblock
\emph{On the Ornstein-Uhlenbeck operator in spaces of continuous functions},
\newblock
J. Funct. Anal., \textbf{131} (1995), no. 1, 94--114.

\bibitem{daprato-lunardi-1}
\newblock
G. Da Prato, and A. Lunardi,
\newblock
\emph{Elliptic operators with unbounded drift coefficients and Neumann boundary condition},
\newblock
J. Differential Equations, \textbf{198}  (2004), 35--52.

\bibitem{daprato-lunardi-2}
\newblock
G. Da Prato, and A. Lunardi,
\newblock
\emph{Ornstein-Uhlenbeck operators with time periodic coefficients},
\newblock
J. Evol. Equ.  \textbf{7}  (2007),  587--614.

\bibitem{DPR}
\newblock
G. Da Prato, and M. R\"ockner,
\emph{Dissipative stochastic equations in Hilbert space with time dependent coefficients},
\newblock
Rend. Lincei Mat. Appl. \textbf{17} (2006), 397--403.

\bibitem{DPR1}
\newblock
G. Da Prato, and M. R\"ockner,
\emph{A note on evolution systems of measures for time-dependent stochastic
differential equations},
\newblock
In: Seminar on Stochastic Analysis, Random Fields and Applications V,  pp. 115–122,
Progr. Probab., \textbf{59}, Birkhäuser, Basel, 2008.

\bibitem{dynkin}
\newblock
E.B. Dynkin,
\newblock
\emph{Three classes of infinite-dimensional diffusions},
\newblock
J. Funct. Anal. \textbf{86} (1989), 75--110.

\bibitem{fornaro-fusco-metafune-pallara}
\newblock
S. Fornaro, N. Fusco, G. Metafune, and D. Pallara,
\newblock
\emph{Sharp upper bounds for the density of some invariant measures},
\newblock
Proc. Roy. Soc. Edinburgh Sect. A  \textbf{139}  (2009),  1145--1161.

\bibitem{geissert-lorenzi-schnaubelt}
\newblock
M. Geissert, L. Lorenzi, and R. Schnaubelt,
\newblock
\emph{$L^p$--regularity for parabolic operators with unbounded
time--dependent coefficients},
\newblock
Annali Mat. Pura Appl.  \textbf{189} (2010), 303-–333.

\bibitem{geissert-lunardi}
\newblock
M. Geissert, and A. Lunardi,
\newblock
\emph{Invariant measures and maximal $L^2$ regularity for nonautonomous Ornstein-Uhlenbeck equations,}
\newblock
J. Lond. Math. Soc., \textbf{77}  (2008),  719--740.

\bibitem{geissert-lunardi-2}
\newblock
M. Geissert, and A. Lunardi,
\newblock
\emph{Asymptotic behavior and hypercontractivity in non-autonomous Ornstein-Uhlenbeck equations,}
\newblock
J. Lond. Math. Soc., \textbf{79}  (2009),  85--106.


\bibitem{ito}
\newblock
S. It{\^o},
\newblock
\emph{Fundamental solutions of parabolic differential
equations and boundary value problems},
\newblock
Jap. J. Math., \textbf{27} (1957),
55--102.

\bibitem{kunze-lorenzi-lunardi}
\newblock
M. Kunze, L. Lorenzi, and A. Lunardi,
\newblock
\emph{Nonautonomous Kolmogorov parabolic equations with unbounded coefficients},
\newblock
Trans. Amer. Math. Soc.  \textbf{362}  (2010), 169--198.

\bibitem{lorenzi-dyn}
\newblock
L. Lorenzi,
\newblock
\emph{Schauder estimates for the Ornstein-Uhlenbeck semigroup in spaces of functions
with polynomial or exponential growth},
\newblock
Dynam. Systems Appl., \textbf{9} (2000), 199--219.

\bibitem{lorenzi-lumer}
\newblock
L. Lorenzi
\newblock
\emph{On a class of elliptic operators with unbounded time- and space-dependent coefficients in $\Bbb R^N$},
\newblock
In Functional analysis and evolution equations,   433--456,
\newblock
Birkh\"auser, Basel, 2008.

\bibitem{lorenzi-DCDS}
\newblock
L. Lorenzi,
\newblock
\emph{Optimal regularity for nonautonomous Kolmogorov equations},
\newblock
Discr. Cont. Dyn. Syst. Series S, \textbf{4} (2011), 169--191.

\bibitem{lorenzi-lunardi}
\newblock
L. Lorenzi, and A. Lunardi,
\newblock
\emph{Elliptic operators with unbounded diffusion coefficients in $L^2$ spaces with respect to invariant measures},
\newblock
J. Evol. Equ.,  \textbf{6}  (2006),  691--709.

\bibitem{lorenzi-lunardi-zamboni}
\newblock
L. Lorenzi, A. Lunardi, and A. Zamboni,
\newblock
\emph{Asymptotic behavior in time periodic parabolic problems with unbounded
coefficients},
\newblock
J. Differential Equations, \textbf{249} (2010), 3377-3418.



\bibitem{lorenzi-zamboni}
\newblock
L. Lorenzi, and A. Zamboni,
\newblock
\emph{Cores for parabolic operators with unbounded coefficients,}
\newblock
J. Differential Equations  \textbf{246}  (2009),  2724--2761.

\bibitem{lunardi-O-U}
\newblock
A. Lunardi,
\emph{On the Ornstein-Uhlenbeck Operator in $L^2$ spaces with respect to invariant measures},
\newblock
Trans. Amer. Math. Soc., \textbf{349} (1997), no. 1, 155--169.

\bibitem{lunardi-studia}
\newblock
A. Lunardi,
\newblock
\emph{Schauder theorems for linear
elliptic and parabolic problems with unbounded coefficients in $\R^N$},
\newblock
Studia Math., \textbf{128} (1998), 171--198.

\bibitem{metafune-O-U}
\newblock
G. Metafune,
\newblock
\emph{{$L^p$}-spectrum of Ornstein-Uhlenbeck operators},
\newblock
Ann. Scuola Norm. Sup. Pisa Cl. Sci., \textbf{30} (2001), no. 1, 97--124.

\bibitem{MPW}
\newblock
G. Metafune, D. Pallara, and M. Wacker,
\newblock
\emph{Feller semigroups on $\R^N$},
\newblock
Semigroup Forum, \textbf{65} (2002), 159--205.

\bibitem{metafune-pallara-enrico}
\newblock
G. Metafune, D. Pallara, and E. Priola,
\newblock
\emph{Spectrum of Ornstein-Uhlenbeck operators in $L^p$ spaces with respect to invariant measures},
\newblock
J. Funct. Anal., \textbf{196} (2002), no. 1, 40--60.

\bibitem{metafune-pallara-rhandi}
\newblock
G. Metafune, D. Pallara, and A. Rhandi,
\newblock
\emph{Global properties of invariant measures},
\newblock
J. Funct. Anal.  \textbf{223}  (2005), 396--424.


\bibitem{metafune-O-U-2}
\newblock
G. Metafune, J. Pr{\"u}ss, A. Rhandi, and R. Schnaubelt,
\newblock
\emph{The domain of the Ornstein-Uhlenbeck operator on an $L^p$-space with invariant measure},
\newblock
Ann. Scuola Norm. Sup. Pisa Cl. Sci., \textbf{1} (2002), no. 2, 471--485.

\bibitem{PRS}
\newblock
J. Pr\"uss, A. Rhandi, and R. Schnaubelt,
\newblock
 \emph{The domain
of elliptic operators on $L^p({\mathbb R}^d)$ with unbounded drift
coefficients},
\newblock
Houston J. Math. \textbf{32} (2006), 563--576.


\bibitem{stannat1}
\newblock
W. Stannat,
\newblock
\emph{Time-dependent diffusion operators on $L^1$},
J. Evol. Equ. \textbf{4} (2004), 463--495.
\end{thebibliography}
\end{document}